\documentclass{amsart}
\usepackage[utf8x]{inputenc}
\usepackage{amsthm}
\usepackage{amsmath}
\usepackage{amssymb}
\usepackage{amsfonts}
\usepackage{amscd}

\newtheorem{theorem}{Theorem}

\newtheorem{corollary}[theorem]{Corollary}

\newtheorem{definition}[theorem]{Definition}
\newtheorem{example}[theorem]{Example}

\newtheorem{proposition}[theorem]{Proposition}
\newtheorem{remark}[theorem]{Remark}

\newcommand{\R}{\mathbb{R}}

\newcommand{\Z}{\mathbb{Z}}
\newcommand{\Q}{\mathbb{Q}}

\newcommand{\T}{\mathbb{T}}

\begin{document}

\title[Block Conjugacy] {Block Conjugacy of Irreducible Toral Automorphisms}
\author[Bakker]{Lennard F. Bakker}
\address{Department of Mathematics, Brigham Young University, Provo, Utah, USA}
\email{bakker@math.byu.edu}
\author[Martins Rodrigues]{Pedro Martins Rodrigues}
\address{Department of Mathematics, Instituto Superior T\'ecnico, Univ. Tec. Lisboa, Lisboa, Portugal}
\email{pmartins@math.ist.utl.pt}

%\thanks{}

\subjclass[2000]{ }
%\date{July 23, 2015}
\commby{}

\begin{abstract} We introduce a relation of block conjugacy for irreducible toral automorphism, and prove that block conjugacy is equivalent to weak equivalence of the ideals associated to the automorphisms. We characterize when block conjugate automorphisms are actually conjugate in terms of a group action on invariant and invariantly complemented subtori, and detail the relation of block conjugacy with a Galois group. We also investigate the nature of the relationship between ideals associated to non-block conjugate irreducible automorphisms.
 \end{abstract}

\maketitle

\section{Introduction}

Topological conjugacy of irreducible hyperbolic toral automorphisms is easily reduced to an algebraic and number theoretic problem. This reduction, however, does not  answer the question of finding a complete set of dynamically significant and computable conjugacy invariants. It seems that results in this direction should result from the study of the connections between the dynamical and the algebraic or number theoretic settings. The relation between algebraic number theory and integer matrices is a classical theme, explored by, among others, Olga Taussky (see \cite{Ta1,Ta2,Ta3,Ta4,Ta5,Ta6,Ta7,Ta8,Ta}). The study of integer matrices in an algebraic number theoretic setting implies the consideration of objects and problems that are usually not thoroughly studied in the theory, namely, ideals in subrings of the rings of algebraic integers in number fields. An important contribution in this direction was the paper by Dade, Taussky, and Zassenhaus \cite{DTZ}. The present paper is, partly, an attempt to present some of the consequences of the results therein, in a form suitable for its use in the dynamical systems context.

In \cite{BMR}, a framework was introduced for the study of the actions of hyperbolic automorphisms in the rational torus, i.e., in the set of periodic points. This consists, briefly, in considering the action of the automorphisms on the profinite completion of $\Z^{n}$, which is the dual group of the direct limit of the family of finite invariant groups of periodic orbits in the torus $T^{n}$. One of the results stated there is that a sufficient condition for conjugacy of the actions of two hyperbolic toral automorphisms on this profinite group, is that the associated ideal classes (see next section) are weakly equivalent.  This condition is also necessary, as is proved in a forthcoming paper.

The main result of this paper is the establishment of the equivalence between the weak equivalence of ideals and an equivalence relation called $2$-block conjugacy between irreducible automorphisms. The first examples of $2$-block conjugacy were produced not many years ago in \cite{MS1}. For these examples, one of the two irreducible automorphisms was a companion matrice, and sufficient conditions on the associated ideals were given. But the main result of this paper supersedes the earlier sufficient conditions and does not require one of the two automorphisms to be a companion matrix.

The paper is organized as follows. The second section of the paper gathers the basic algebraic concepts and results needed for the sequel. The third section defines and gives the basic theory of the relation of $k$-block conjugacy for an integer $k\geq 2$. The fourth section proves the main result that weak equivalence of ideals is equivalent to $2$-block conjugacy of automorphisms. The fifth section gives a characterization of when $2$-block conjugate automorphisms are actually conjugate in terms of a group action on invariant and invariantly complemented subtori, and details a connection of $2$-block conjugacy with a Galois group.  The sixth section investigates the nature of the relationship between ideals associated to automorphisms that are not $2$-block  conjugate.

\section{Toral automorphisms and Ideal classes}

For a fixed canonical basis in $\R^{n}$, an automorphism of the torus $\T^{n}=\R^{n}/\Z^{n}$ is represented by an invertible integer matrix, i.e., an element $A \in GL_{n}(\Z)$. The automorphism is irreducible if and only if the characteristic polynomial $f(t)$ of the matrix is irreducible over $\Z$; it is hyperbolic if and only if $f(t)$ is hyperbolic, i.e., if it has no roots of absolute value one. Two automorphisms are topologically conjugate if and only if the corresponding matrices $A$ and $B$ are algebraically conjugated in this group, i.e., if and only if there exists $P \in GL_{n}(\Z)$ such that $P A=B P$ (see \cite{AP}). We denote this relation as $A \approx B$. The characteristic polynomial is an obvious conjugacy invariant that determines, and is determined by, the two fundamental invariants of topological entropy and the Artin-Mazur zeta function (see \cite{Rob}).

We fix an irreducible polynomial $f(t)$ of degree $n$. This determines an algebraic number field $K=\Q[t]/(f(t))$. The class of $t$, which we denote from now on by $\beta$, is a simple root of $f$ in $K$ and we identify $K$ with the isomorphic field $\Q(\beta)$. The notation $O_{K}$ denotes the ring of all algebraic integers of $K$. An automorphism $A \in GL_{n}(\Z)$ with characteristic polynomial $f(t)$ acts on $K^{n}$ and so there exists a right eigenvector $u \in K^{n}$ associated to $\beta$, namely, $A u=\beta u$, where the entries $u_{i}$ of $u$ determine a $\Z[\beta]$-fractional ideal $I=u_1\Z + \cdots + u_n\Z$. The action of $A$, on the right, on $\Z^{n}$, is thus identified with multiplication by $\beta$ in $I$, namely, if $x= m \cdot u $ for some $m \in \Z^{n}$, then $m A u=\beta x$. The fact that the entries of $u$ are a basis of the free $\Z$-module $I$ has the obvious but important consequence that if $m \cdot u=0$ then $m=0$, and so, if $M \in M_n(\Z)$ (an $n\times n$ integer matrix) and $M u=0$ then $M=0$. 

\begin{definition}Two fractional ideals $I$ and $J$ are arithmetically equivalent if $J=\alpha I$ for some $\alpha \in K$.
 \end{definition}
 
Each $A \in GL_{n}(\Z)$ is in fact associated with a class of ideals. On the other hand, if $A$ and $B$ are conjugate they are associated with the same ideal class, the conjugating matrix $P$ acting as a change of basis of the ideals, taken as $\Z$-modules. This is the the essence of the following result found in \cite{Ta1,Ta}.

\begin{theorem}[Latimer-MacDuffee-Taussky]
There is a bijection between the conjugacy classes of $GL_{n}(\Z)$ matrices with irreducible characteristic polynomial $f(t)$ and arithmetic equivalence classes of ideals of $\Z[\beta]$.
\end{theorem}

For two fractional ideals $I$ and $J$ of $\Z[\beta]$, we define another fractional ideal
\[(J:I)=\{x \in K: xI \subset J\}.\]
The coefficient ring of $I$ is defined to be $(I:I)$, which is always an order of the field, i.e., a subring of $O_{K}$ that has rank $n$ as a $\Z$-module. Obviously $(I:I)\supset \Z[\beta]$. Suppose that $A \in GL_{n}(\Z)$ represents, as above, multiplication by
$\beta$ in $I$, with respect to a certain basis over $\Z$; then, if $\theta \in K$ is given by a polynomial $\theta=p(\beta)$ with rational coefficients, then $\theta \in (I:I)$ iff $p(A)$ is an integer matrix. An obvious necessary condition for two ideals to be arithmetically equivalent is that they have the same coefficient ring.
An ideal $X$ with coefficient ring $R$ is invertible if $X(R:X)=R$. An invertible ideal may always be generated, as an $R$-module, by two (or one) elements.

\begin{definition} Two ideals $I$ and $J$ of $\Z[\beta]$ are weakly equivalent if there exist ideals $X$ and $Y$ and an order $R$ such that
\[ I X= J,\ J Y= I,\ X Y=R.\]
 \end{definition}
 
The definition of weak equivalence implies that $I$ and $J$ have the same coefficient ring $R$. A detailed exposition on the weak equivalence relation may be found in \cite{DTZ}. A few facts are immediate consequences of the definition: arithmetically equivalent ideals are always weak equivalent, so the weak equivalence relation contains arithmetic equivalence. On the other hand, two invertible but not arithmetically equivalent ideals, in an order $R$, are always weakly equivalent. Furthermore, an invertible ideal and a non-invertible ideal, both in the same order $R$, are not weakly equivalent. 

We recall an example found in \cite{DTZ} that gives by a general construction, the existence of non-invertible ideals. Let $\theta$ be any algebraic integer of degree $n>2$, and set $R_{0}$ to be the monogenic order $\Z[\theta]=\Z+\Z \theta+\Z \theta^{2}+\cdots+\Z \theta^{n-1}$. The subring
\[I=\Z+\Z \theta+2\Z \theta^{2}+\cdots+2\Z \theta^{n-1}\]
is a fractional ideal for the suborder $R=\Z+2\Z \theta+2\Z \theta^{2}+\cdots+2\Z \theta^{n-1}$ of $R_{0}$. Here we have $R=(I:I)$ because $1\in I$ implies $(I:I) \subset I$, and because $\theta \notin (I:I)$. On the other hand, for any $0<k<n-1$, we have
\[I^{k}=\Z+\Z \theta+\cdots +\Z\theta^{k}+2\Z \theta^{k+1}+\cdots+2\Z \theta^{n-1}\]
while $I^{n-1}=R_{0}$. Since $I R_{0}=I^{2}R_{0}= \cdots=I^{n-2}R_{0}=R_{0}$, we conclude that $I$ (as well as for each power $I^{k}$ for $0<k<n-1$) is not invertible, because if $I J=R$ then we would have
\[I R_{0}=R_{0} \implies I J R_{0}=J R_{0} \implies R R_{0}=J R_{0} \implies R_{0}=J R_{0}\implies\]
\[\implies R_{0}=I^{n-1}J=I^{n-2}R=I^{n-2},\]
a contradiction.

\section{Semiconjugacy and Block Conjugacy}

Semiconjugacy by continuous surjection fails to distinguish two toral automorphisms $A$ and $B$ on ${\mathbb T}^n$ with the same irreducible characteristic polynomial $f(t)$. The components of right eigenvectors $u$ of $A$ and $v$ of $B$ corresponding to $\beta$ both form bases of $K$ as a $\Q$-vector space. So there is $P\in GL_n(\Q)$ such that $u = Pv$. Then $(AP-PB)v = Au-\beta u=0$, so that $AP=PB$. Hence there are integer matrices $X$ and $Y$ given by appropriate integer multiples of $P$ of $P^{-1}$ respectively, that give continuous surjections of ${\mathbb T}^n$ that satisfy $AX=XB$ and $BY = YA$. These state that $B$ is semiconjugate to $A$ by the continuous surjection $X$, and that $A$ is semiconjugate to $B$ by the continuous surjection $Y$.

For an integer $k\geq 2$, semiconjugacy by continuous surjection also fails to distinguish $\bigoplus_{i=1}^k B$ from $A$, and $\bigoplus_{i=1}^k A$ from $B$. Here $\bigoplus_{i=1}^k B$ is the diagonal action of $B$ on $({\mathbb T}^n)^k$, i.e.,
\[ \bigoplus_{i=1}^k B = \begin{pmatrix} B & 0 & 0 & \dots & 0 \\ 0 & B & 0 & \dots & 0 \\ 0 & 0 & B & \dots & 0 \\ \vdots & \vdots & \vdots & \ddots & \vdots \\ 0 & 0 & 0 & \dots & B\end{pmatrix}.\]
With $\Pi_1$ as the projection of $({\mathbb T}^n)^k$ to the first factor, the continuous surjection $X\circ \Pi_1$ from $({\mathbb T}^n)^k$ to ${\mathbb T}^n$ satisfies $X\circ \Pi_1\circ \bigoplus_{i=1}^k B = A\circ X\circ \Pi_1$. Thus $\bigoplus_{i=1}^k B$ is semiconjugate to $A$ by the continuous surjection $X\circ \Pi_1$. Similarly we also have $Y\circ \Pi_1 \circ \bigoplus_{i=1}^k A = B\circ Y\circ \Pi_1$.

However, semiconjugacy by continuous injection may distinguish $A$ from $\bigoplus_{i=1}^k B$, and/or $B$ from $\bigoplus_{i=1}^k A$. The existence of a continuous injection $\iota:{\mathbb T}^n \to ({\mathbb T}^n)^k$ which satisfies $\iota \circ A = \bigoplus_{i=1}^k B \circ \iota$ implies that the continuously embedded copy of ${\mathbb T}^n$ given by $\iota({\mathbb T}^n)$ is an invariant set for $\bigoplus_{i=1}^k B$, and that $\bigoplus_{i=1}^k B$ induces $A$ on $\iota( {\mathbb T}^n)$. Similar statements hold for a semiconjugacy from $B$ to $\bigoplus_{i=1}^k A$ by a continuous injection.

More can be said when the continuous injection $\iota$, in $\iota \circ A = \bigoplus_{i=1}^k B\circ \iota$, is a linear embedding. In this case, $\iota$ is given by a $nk \times n$ integer matrix of the block form
\[ \begin{pmatrix} M_1 \\ M_3\end{pmatrix}\]
where $M_1$ is $n\times n$ and $M_3$ is $n(k-1)\times n$, for which the greatest common divisor of the $nk$  integers in each column of the $nk\times n$ matrix is $1$. The semiconjugacy $\iota\circ A = \bigoplus_{i=1}^k B \circ \iota$ becomes
\[ \begin{pmatrix} M_1 \\ M_3\end{pmatrix} A = \bigoplus_{i=1}^k B \circ \begin{pmatrix} M_1 \\ M_3\end{pmatrix}.\]
By the Hermite Normal Form \cite{Coh}, the linear embedding can be extended to a $GL_{nk}(\Z)$ matrix
\[ M = \begin{pmatrix} M_1 & M_2 \\ M_3 & M_4\end{pmatrix}.\]
Setting
\[ \begin{pmatrix} S \\ A^\prime\end{pmatrix} = M^{-1}\circ \bigoplus_{i=1}^k B \circ \begin{pmatrix} M_2 \\ M_4\end{pmatrix} \]
gives the equation
\[ \bigoplus_{i=1}^k B \circ M = M \begin{pmatrix} A & S \\ 0 & A^\prime \end{pmatrix}.\]
This implies that 
\[ \hat A = \begin{pmatrix} A & S \\ 0 & A^\prime \end{pmatrix} \in GL_{nk}(\Z),\]
and hence there is a conjugacy between $\bigoplus_{i=1}^k B$ and $\hat A$. On the other hand, if there is $M\in GL_{nk}(\Z)$ and a $GL_{nk}(\Z)$ matrix of the form $\hat A$ that satisfy
\[ \bigoplus_{i=1}^k B \circ M = M \hat A,\]
then $A$ is semiconjugate to $\bigoplus_{i=1}^k B$ by the linear embedding $\iota = M\circ \iota_1$, where $\iota_1 : {\mathbb T}^n \to ({\mathbb T}^n)^k$ is defined by $\theta \to (\theta,0,\dots,0)$; that is
\[  M\circ i_1 \circ A = \bigoplus_{i=1}^k B \circ M \circ i_1.\]
Thus we have proved

\begin{proposition}\label{nkconjugacy} The existence of a semiconjugacy from $A$ to $\bigoplus_{i=1}^k B$ by a linear embedding is equivalent to the existence of a conjugacy between a $GL_{nk}(\Z)$ matrix of the form
\[ \hat A = \begin{pmatrix} A & S \\ 0  & A^\prime\end{pmatrix}\]
and $\bigoplus_{i=1}^k B$.
\end{proposition}

A similar equivalence holds when $A$ and $B$ are interchanged.

The simplest structure of the $GL_{nk}(\Z)$ matrix of the form $\hat A$ occurs when $S=0$ and $A^\prime$ is a block diagonal matrix. Specifically, $A^\prime = \bigoplus_{i=2}^k A_i$ with each $A_i$ belonging to $GL_n(\Z)$. The conjugacy of Proposition \ref{nkconjugacy} then relates the diagonal action of $B$ with another diagonal action $\bigoplus_{i=1}^k A_i$ where $A_1=A$. A similar statement holds for a conjugacy between the diagonal action of $A$ with another diagonal action $\bigoplus_{i=1}^k B_i$ where $B_1=B$ and $B_2,\dots,B_k$ belong to $GL_n(\Z)$.

\begin{definition} For an integer $k\geq 2$, two $GL_n(\Z)$ matrices $A$ and $B$ with the same irreducible characteristic polynomial $f(t)$ are $k$-block conjugate if there exist matrices $A_2,\dots,A_k, B_2,\dots,B_k\in GL_n(\Z)$ such that $\bigoplus_{i=1}^k A \approx \bigoplus_{i=1}^k B_i$, and $\bigoplus_{i=1}^k B \approx \bigoplus_{i=1}^k A_i$. We say $A$ and $B$ are block conjugate if they are $k$-block conjugate for some $k\geq 2$.
\end{definition}

The relation of $k$-block conjugacy of $A$ and $B$ has several properties. It is reflexive. It is also symmetric, but a $GL_{nk}(\Z)$ matrix that conjugates $\bigoplus_{i=1}^k A$ with $\bigoplus_{i=1}^k B_i$ need not be the inverse of a $GL_{nk}(\Z)$ matrix that conjugates $\bigoplus_{i=1}^k B$ with $\bigoplus_{i=1}^k A_i$. (We will show in the next section that the existence of one of the two conjugacies implies the other.) Furthermore, the irreducibility of characteristic polynomial $f(t)$ of $A$ and $B$ implies that the characteristic polynomial of every one of $A_2,\dots,A_k$ and of $B_2,\dots,B_k$ is also $f(t)$. We will shown that $2$-block conjugacy is transitive, and hence an equivalence relation. 

The relation of $k$-block conjugacy is a generalization of conjugacy, which may be seen as $1$-block conjugacy, that presents an enlarged stage for distinguishing and comparing the dynamics  of $A$ and $B$. It has a simple interpretation: if $A$ and $B$ are $k$-block conjugate, for some $k$, then $A$ is conjugated to the restriction of $\oplus_{i=1}^{k} B$ to an embedded $n$-dimensional torus in $\mathbb{T}^{k n}$ and vice-versa. As a special case, it may even happen that, for some $k>1$, $\bigoplus_{i=1}^{k} A \approx \bigoplus_{i=1}^{k} B$ with $A$ and $B$ nonconjugate. We will show in the next section that there exist nonconjugate $A$ and $B$ for which $A\oplus A\approx B\oplus B$.

The relation of $k$-block conjugacy has the following property.

\begin{proposition}\label{kimpliesk+1} If $A$ and $B$ are $k$-block conjugate for some $k\geq 1$, then $A$ and $B$ are $(k+1)$-block conjugate.
\end{proposition}

\begin{proof} Suppose there is $M\in GL_{nk}(\Z)$ such that $\bigoplus_{i=1}^k B \circ M = M\circ \bigoplus_{i=1}^k A_i$. Set $A_{k+1} = B$, and
\[ \tilde M = \begin{pmatrix} M & 0 \\ 0 & I_n\end{pmatrix} \in GL_{n(k+1)}(\Z),\]
where $I_n$ is the $n\times n$ identity matrix. It follows that $ \bigoplus_{i=1}^{k+1} B \circ \tilde M = \tilde M \circ \bigoplus_{i=1}^{k+1} A_i$. A similar extension of $\bigoplus_{i=1}^k A \approx \bigoplus_{i=1}^k B_i$ to $k+1$ holds as well.
\end{proof}

We could have two automorphisms which are $k$-block conjugate for some $k\geq 3$, but not $l$-block conjugate for all $l=1,\dots,k-1$. But this can not happen, because we will prove in the next section the remarkable property that if $A$ and $B$ are $k$-block conjugate for some $k\geq 3$, then they are $2$-block conjugate. This combined with Proposition \ref{kimpliesk+1} implies there are then only three possibilities for two automorphism: conjugate, $2$-block conjugate but not conjugate, or not block conjugate.

The existence of a $k$-block conjugacy for $k\geq 2$ imposes necessary conditions on the conjugating matrices. Suppose a $GL_{nk}(\Z)$ matrix $M$ satisfies $\bigoplus_{i=1}^k B \circ M = M \circ \bigoplus_{i=1}^k A_i$. If $X$ is an $n\times n$ matrix in the partition of $M$ into $n\times n$ submatrices, then $\bigoplus_{i=1}^k B \circ M = M \circ\bigoplus_{i=1}^k A_i$ implies that $X$ satisfies $BX = XA_i$ for some $i$. This says that $A_i$ is semiconjugate to $B$ by the endomorphism $X$. Thus $M$ is built with $n\times n$ matrices belonging to the sets
\[\Lambda(B,A_i)= \{X \in M_n(\Z): BX = X A_i \},\ i=1,\dots,k.\]
Similarly, a $GL_{nk}(\Z)$ matrix $N$ that satisfies $\bigoplus_{i=1}^k A \circ N = N \circ \bigoplus_{i=1}^k B_i$, is built from $n\times n$ matrices belonging to the sets $\Lambda(A,B_i)$, $i=1,\dots,k$.

Each of the sets $\Lambda(B,A_i)$, $\Lambda(A,B_i)$ is an additive abelian group that has a module structure. We will describe the module structure for $\Lambda(A,B)$, as the others are similar. Associated to $A$ and $B$ are the $\Z[\beta]$-fractional ideals $I=\bigoplus_{i} u_{i}\Z$ and  $J=\bigoplus_{i} v_{i}\Z$ generated respectively by the entries of their right eigenvectors $u$ and $v$ corresponding to $\beta$. The following then follows directly from the definition of $\Lambda(A,B)$.

\begin{proposition}\label{Rmodule} For the ring $R= (I:I)\cap (J:J)$, the abelian group $\Lambda(A,B)$ is a left and a right $R$-module with multiplications defined by
\[R \times \Lambda(A,B) \rightarrow \Lambda(A,B),  \hspace{0.5 cm} (p(\beta),X) \rightarrow p(A) X,\] 
\[R \times \Lambda(A,B)  \rightarrow \Lambda(A,B), \hspace{0.5 cm} (p(\beta),X) \rightarrow X p(B),\]
where $p(x)\in \Q[x]$ satisfies $p(\beta)\in R$. Moreover, there is an isomorphism of $R$-modules
 \[\varphi:\Lambda(A,B) \rightarrow (J:I)=\{\theta \in K: \theta I \subset J\}\]
given by \[\varphi(X)=\theta \, \mbox{  where }  X v = \theta u.\]
\end{proposition} 

We show that each endomorphism in $\Lambda(A,B)$ is either the $0$ matrix or surjective. Recall that there is $P\in GL_n(\Q)$ such that $u=Pv$. Combining this with $Xv = \theta u$ gives $P^{-1}Xu = \theta u$. The determinant of $P^{-1}X$ is the absolute norm of $\theta \in K$, which determinant is equal to zero if and only if $\theta = 0$. This implies

\begin{corollary}\label{detnot0}
 If $X \in \Lambda(A,B)$ then $\det X=0$ if and only if $X=0$.
\end{corollary}

We will tacitly use Proposition \ref{Rmodule} and Corollary \ref{detnot0}  in upcoming proofs.

\section{Weak Equivalence and Block Conjugacy}

We establish some characterizations of weak equivalence of ideals in terms of $k$-block conjugacy of the associated irreducible toral automorphisms. We start with

\begin{proposition}\label{weakequivalence}
Suppose $I$ and $J$ are ideals associated respectively to $A,B\in GL_n({\mathbb Z})$ both having the same irreducible characteristic polynomial $f(t)$. If $I$ and $J$ are weakly  equivalent, then $A$ and $B$ are $2$-block conjugate.
\end{proposition}

\begin{proof} Suppose $I$ and $J$ are weakly equivalent ideals. Then they have the same ring of coefficients $R$, and there exist ideals $X$ and $Y$ such that $IX = J$, $JY=I$, and $XY=R$. Since $X$ and $Y$ are invertible ideals, there exist generators $a_1,a_2,b_1,b_2\in R$ such that
\[ X = b_1 R + b_2 R,\ Y = a_1 R + a_2 R.\]
Because $XY=R$, we can choose the generators so that $a_1b_1+a_2b_2=1$.

Let $\Z$-bases of $I$ and $J$ be given by the components of $u$ and $v$ respectively, where $Au=\beta u$ and $Bv=\beta v$. Since $a_i\in Y$ and $JY=I$, there exist $M_{i1}\in M_n(\Z)$ such that $M_{i1} u = a_iv$. These $M_{i1}\in \Lambda(B,A)$ because
\[ BM_{i1} u = a_i\beta v = \beta a_i v = M_{i1}Au.\]
Since $b_i\in X$ and $IX=J$, there exist $W_{1j}\in M_n{\Z}$ such that $W_{1j}v = b_j u$. These $W_{1j}\in \Lambda(A,B)$ because
\[ AW_{1j} = b_j \beta u = \beta b_j u = W_{1j}Bv.\]
The matrices $M_{i1}$ and $W_{1j}$ satisfy
\[ (W_{11}M_{11} + W_{12}M_{21})v = (a_1b_1+a_2b_2)v = v,\]
and so
\[ W_{11}M_{11} + W_{12}M_{21} = I_n.\]
(In a similar manner it follows that $M_{11}W_{11}+M_{21}W_{12}=I_n$.)

Fix a $\Z$-basis $w=(w_1,\dots,w_n)^T$ for the ideal $JX$. Define $A^\prime\in GL_n(\Z)$ by $A^\prime w = \beta w$. Since $b_i\in X$ and the components of $w$ are in $JX$, then there are $M_{i2}\in M_n(\Z)$ such that
\[ M_{12} w = -b_2 v, \ M_{22} w = b_1 v.\]
The matrix $M_{12}\in \Lambda(B,A^\prime)$ because
\[ BM_{12} w = -b_2\beta v = \beta(-b_2 v) = M_{12}A^\prime w,\]
and the matrix $M_{22}\in \Lambda(B,A^\prime)$ because
\[ BM_{22} w = b_1 \beta v = \beta b_1 v = M_{22} A^\prime w.\]
Since $a_i\in Y$ and the components of $w$ are in $JX$, the components of $a_i w$ are in $JXY = JR = J$. So there are $W_{2j}\in M_n(\Z)$ such that
\[ W_{21} v = -a_2 w, \ W_{22} v = a_1 w.\]
The matrix $W_{21}\in \Lambda(A^\prime, B)$ because
\[ A^\prime W_{21} v = -a_2\beta w = \beta(-a_2 w) = W_{21}B v,\]
and the matrix $W_{22}\in \Lambda(A^\prime,B)$ because
\[ A^\prime W_{22} v = a_1 \beta w = \beta a_1w = W_{22}Bv.\]
The matrices $M_{i2}$ and $W_{2j}$ satisfy
\[ (W_{21}M_{12} + W_{22}M_{22}) w = (a_2b_2 + a_1b_1)w = w,\]
and so
\[ W_{21}M_{12} + W_{22}M_{22} = I_n.\]
(In a similar manner it follows that $M_{12}W_{21} + M_{22}W_{22}=I_n$.)

By construction we have $(B\oplus B)M = M(A\oplus A^\prime)$ for
\[ M = \begin{pmatrix} M_{11} & M_{12} \\ M_{21} & M_{22} \end{pmatrix}\in M_{2n}(\Z).\]
Also by construction the matrix
\[ W = \begin{pmatrix} W_{11} & W_{12} \\ W_{21} & W_{22}\end{pmatrix}\]
satisfies
\[ WM = \begin{pmatrix} I_n & * \\ * & I_n\end{pmatrix}.\]
If the starred entries are both $0$, then $M\in GL_{2n}(\Z)$. To this end we have
\[ (W_{11}M_{12} + W_{12}M_{22})w = -b_1b_2 u + b_1b_2 u = 0,\]
so that $W_{11}M_{12} + W_{12}M_{22} = 0$, and we have
\[ (W_{21}M_{11}+W_{22}M_{21}) u = -a_1a_2 w + a_1a_2 w = 0,\]
and so $W_{21}M_{11}+W_{22}M_{21} = 0$. Thus $B\oplus B \approx A\oplus A^\prime$.

%\begin{proof}
%Let $X=x_{1}R+x_{2}R$ and $Y=y_{1}R+y_{2}R$ be invertible $R$-ideals satisfying $I X=J$, $J Y=I$, and $X Y=R$. Because $XY=R$, we can choose the generators such that  $x_{1}y_{1}+x_{2}y_{2}=1$. Define matrices $X_{i} \in \Lambda(A,B)$ and $Y_{i} \in \Lambda(B,A)$ by $x_{i}u=X_{i}v$ and $y_{i}v=Y_{i}u$. Because the entries of $u$ and $v$ are $\Z$-bases for $I$ and $J$ respectively, and
%\begin{align*}
%(X_1Y_1 + X_2Y_2)u & = (x_1y_1+x_2y_2)u=u,\\
%(Y_1X_1 + Y_2X_2)v & = (y_1x_1+y_2x_2)v = v,
%\end{align*}
%the matrices $X_1,X_2,Y_1,Y_2$ satisfy
%\[X_{1}Y_{1}+X_{2}Y_{2}=Y_{1}X_{1}+Y_{2}X_{2}=I_{n},\]
%where $I_{n}$ denotes the $n\times n$ identity matrix.
 
%Consider the ideal $J X=w_{1}\Z+\cdots+w_{n}\Z$ and let $A'$ be the  $GL(n,\Z)$ matrix defined by $A' w=\beta w$. As before, there are matrices $T_{i} \in \Lambda(B,A')$ and $S_{i} \in \Lambda(A',B)$ such that $x_{i}v=T_{i}w$ and $y_{i}w=S_{i}v$ satisfying similar relations. We then have
% \[ \begin{pmatrix} B & 0 \\ 0 & B\end{pmatrix} \begin{pmatrix} Y_1 & -T_2 \\ Y_2 & T_1 \end{pmatrix} = \begin{pmatrix} Y_1 & -T_2 \\ Y_2 & T_1\end{pmatrix} \begin{pmatrix} A & 0 \\ 0 & A^\prime\end{pmatrix},\]
%and
%\[ \begin{pmatrix} X_1 & X_2 \\ -S_2 & S_1 \end{pmatrix} \begin{pmatrix} Y_1 & -T_2 \\ Y_2 & T_1\end{pmatrix} = I_{2n}.\]
%Thus $B\oplus B \approx A\oplus A^\prime$.
 
Reversing the roles of $I$ and $J$ and of $A$ and $B$ and introducing the ideal $I Y$ and a corresponding matrix $B'$, we obtain $A\oplus A \approx B\oplus B^\prime$. Thus $A$ and $B$ are $2$-block conjugate. 
\end{proof}

We can interpret the $GL_{2n}(\Z)$ matrix in the proof of Proposition \ref{weakequivalence} that conjugates $B\oplus B$ with $A\oplus A^\prime$ in terms of an $R$-module isomorphism between direct sums of ideals. On $\Z^n \oplus \Z^n$, the conjugating matrix is the isomorphism,
\[(m,l)\rightarrow (m,l) \begin{pmatrix} M_{11} & M_{12} \\ M_{21} & M_{22}\end{pmatrix}.\]
The $R$-module structure on $\Z \oplus \Z$ is the right action $p(B) \cdot (m,l) = (m\, p(B), l\, p(B))$ where $p(\beta)\in R$. In relating $\Z \oplus \Z$ to a direct sum of ideals, we have the $R$-module isomorphism
\[\Z^n \oplus \Z^n \rightarrow J \oplus J \mbox{ defined by } (m,l) \rightarrow (m \cdot v, l \cdot v),\]
where the right action of $R$ in $\Z^n \oplus \Z^n$ is defined by  $\alpha \cdot(m,l)=(m\, p(B),l \, p(B))$ for $\alpha=p(\beta)\in R$. We also have the $R$-module isomorphism
\[\Z^n \oplus \Z^n \rightarrow J Y \oplus J X \mbox{ defined by } (m,l) \rightarrow (m \cdot u, l \cdot w),\]
where the right action of $R$ in $\Z^n \oplus \Z^n$ is defined by $\alpha \cdot(m,l)=(m\, p(A),l\, p(A'))$. Through these $R$-module isomorphisms, the $GL_{2n}(\Z)$ matrix conjugating $B\oplus B$ with $A\oplus A^\prime$ induces the $R$-module isomorphism
\[\phi:J \oplus J \rightarrow  J Y \oplus J X \mbox{  defined by  } \phi(x,y)=(a_{1} x+a_{2} y, -b_{2} x+b_{1} y),\]
where $X=b_1R+b_2R$ and $Y=a_1R+a_2R$. 

%The inverse of $\phi$ is the $R$-module isomorphism
%\[\phi^{-1}:J Y \oplus J X\rightarrow  J  \oplus J  \mbox{  defined by  } \phi^{-1}(c,d)=(x_{1} c-y_{2} d, x_{2} c+y_{1} d).\] 

\begin{example}{\rm We illustrate the construction of a $2$-block conjugacy for the nonconjugate $GL_2({\mathbb Z})$ matrices
\[ A = \begin{pmatrix} 8 & 5 \\ 3 & 2\end{pmatrix}, \  B=\begin{pmatrix} 9 & 8 \\ 1 & 1\end{pmatrix},\]
each having the irreducible $f(t) = t^2 - 10 t +1$ as its characteristic polynomial. For $\beta$ the largest real root of $f(t)$, eigenvectors of $A$ and $B$ are respectively
\[ u = \begin{pmatrix} \beta -2 \\ 3\end{pmatrix}, \ v = \begin{pmatrix} \beta -1 \\ 1\end{pmatrix}.\]
For the order $R = {\mathbb Z}[\beta]$, the ideals associated to $A$ and $B$ are
\[ I = (\beta -2)R + 3R, \ J = (\beta-1)R + R = R.\]
Here $B$ is conjugate to the companion matrix for $f(t)$. Also the ideal $I$ is invertible in $R$, so that $II^{-1} = R$. The ideals $I$ and $J$ are weakly equivalent because for
$Y=I$ and $X = I^{-1} = Y^{-1}$, we have
\[ XY = R, \ JY = JI = RI = I, \ IX = I I^{-1} = R = J.\]
For $a_1 = \beta -2$, $a_2 = 3$, the matrices $M_{i1}\in M_n(\Z)$ defined by $a_i u = M_{i1} v$ are
\[ M_{11} = \begin{pmatrix} 7 & 5 \\ 1 & 0\end{pmatrix}, \ M_{21} = \begin{pmatrix}3 & 1 \\ 0 & 1\end{pmatrix}.\]
Each $M_{i1}$ belongs to $\Lambda(B,A)$. To find $b_1$ and $b_2$ such that $X= b_1 R + b_2 R$ and $a_1b_1 + a_2b_2 =1$, we find an element of $X = I^{-1}$. One such element is 
$b_1 = -(1/3)(\beta + 1)$. Then solving $a_1b_1 + a_2b_2 =1$ for $b_2$ gives $b_2 = \beta$. The integer matrices defined by $W_{1j} v = b_j u$ are
\[ W_{11} = \begin{pmatrix} -3 & -2 \\ -1 & -2\end{pmatrix},\ W_{12} = \begin{pmatrix} 8 & 7 \\ 3 & 3\end{pmatrix}.\]
Each $W_{1j}$ belongs to $\Lambda(A,B)$. The $M_{i1}$ and $W_{1j}$ satisfy $W_{11}M_{11} + W_{12}M_{21} = I_2$.

To get $(B\oplus B) \approx (A\oplus A^\prime)$ for some $A^\prime$, we consider the ideal $JX = RX = X$ which has a ${\mathbb Z}$-basis
of $-(1/3)(\beta + 1)$ and $\beta$. Set
\[ w = \begin{pmatrix} -(1/3)(\beta +1) \\ \beta\end{pmatrix}.\]
The $GL_2({\mathbb Z})$ matrix $A^\prime$ defined by $A'w = \beta w$ is
\[ A' = \begin{pmatrix} -1 & -4 \\ 3 & 11\end{pmatrix}.\]
The characteristic polynomial of $A'$ is $f(t)$. The integer matrices $M_{i2}$ defined by $M_{12} w = -b_2 v$ and $M_{22} w = b_1 v$ are
\[ M_{12} = \begin{pmatrix} -3 & -10 \\ 0 & -1\end{pmatrix},\ M_{22} = \begin{pmatrix} -2 & -4 \\ 1 & 0 \end{pmatrix}.\]
Each $M_{i2}$ belongs to $\Lambda(B,A^\prime)$. The integer matrices $W_{2j}$ defined by $W_{21} v = -a_2 w$ and $W_{22}v = a_1w$ are
\[ W_{21} = \begin{pmatrix} 1 & 2 \\ -3 & -3\end{pmatrix},\ W_{22} = \begin{pmatrix} -3 & -2 \\ 8 & 7\end{pmatrix}.\]
%\[ S_1 = \begin{pmatrix} -3 & -2 \\ 8 & 7\end{pmatrix}, \ S_2 = \begin{pmatrix} -1 & -2 \\ 3 & 3\end{pmatrix}.\]
Each $W_{2j}$ belongs to $\Lambda(A',B)$. The integer matrices $M_{ij}$ and $W_{ij}$ satisfy
\[ \begin{pmatrix} W_{11} & W_{12} \\ W_{21} & W_{22}\end{pmatrix} \begin{pmatrix} M_{11} & M_{12} \\ M_{21} & M_{22}\end{pmatrix} = I_{4},\]
and so
\[ M = \begin{pmatrix} M_{11} & M_{12} \\ M_{21} & M_{22}\end{pmatrix}
= \begin{pmatrix} 7 & 5 & -3 & -10 \\ 1 & 0 & 0 & -1 \\ 3 & 1 & -2 & -4 \\ 0 & 1 & 1 & 0 \end{pmatrix} \in GL_{4}({\mathbb Z}).\]
For this $M$ we have $(B\oplus B)M = M(A\oplus A^\prime)$, so that $B\oplus B \approx A\oplus A'$.

Similarly, we obtain $(A\oplus A) \approx (B\oplus B^\prime)$ for some $B^\prime$ by considering the ideal $IY= I^2 = 3R$ which has a $\Z$-basis of $3$ and $3\beta$. From this we get
\[ B^\prime = \begin{pmatrix} 0 & 1 \\ -1 & 10\end{pmatrix} \in GL_2(\Z),\]
which is the companion matrix for $f(t)$. The $GL_4(\Z)$ matrix
\[ N = \begin{pmatrix} -3 & -2 & 2 & -1 \\ -1 & -2 & -3 & 0 \\ 8 & 7 & 1 & 2 \\ 3 & 3 & -2 & 1\end{pmatrix}\]
satisfies $(A\oplus A)N = N(B\oplus B^\prime)$. Thus $A$ and $B$ are $2$-block conjugate.

Note that the conjugating matrices $M$ and $N$ are not inverses of each other. In particular, they have different determinants, with ${\rm det}(M) = 1$ and ${\rm det}(N) = -1$.
}
\end{example}

\begin{remark}{\rm 
 In the above example, the automorphism $A'$ is conjugate to $A$, and the same thing happens with $B$ and $B'$. Thus $A\oplus A$ and  $B \oplus B$ are conjugate while $A$ and $B$ are not. The conjugacy of $A$ and $A^\prime$ and of $B$ and $B^\prime$ will always happen when, as in this case, every element in the group of arithmetic equivalence classes of invertible ideals has order two. 
}\end{remark}

We investigate the consequences of assuming one of the two conjugacies in a $k$-block conjugacy for some $k\geq 2$. Suppose $A=A_{1},A_{2},\cdots, A_{k}$ and $B$ are automorphisms such that
\[\bigoplus_{i=1}^{k}B \circ M=M \circ \bigoplus_{i=1}^{k}A_{i}\]
for some $M \in GL_{n k}(\mathbb{Z})$. If we partition $M$ into $n\times n$ blocks $M_{ij}$, $i,j=1,\dots,k$, then $M_{i1}\in\Lambda(B,A)$. If $B$ is associated to a $\Z[\beta]$-ideal $J$, generated by the entries of a column eigenvector $v$ and $A$ is associated to a $\Z[\beta]$-ideal $I$, generated by the entries of a column eigenvector $u$, there exist $a_{i} \in (I:J)$, $i=1,\dots,k$, such that $M_{i 1}u=a_{i} v$. Similarly, if we partition $W = M^{-1}$ into $n\times n$ blocks $W_{ij}$, $i,j,=1,\dots,k$, we have $W_{1j}\in \Lambda(A,B)$ for every $j=1,\dots,k$. So there exist $b_j\in (J:I)$, $j=1,\dots,k$, such that $W_{1 j}v=b_{j}u$. The matrix equation $\sum_{i=1}^{k}W_{1i}M_{i1}=I_{n}$, implies that $\sum_{i=1}^{k}a_{i}b_{i}u=u$ and so $\sum_{i=1}^{k}a_{i}b_{i}=1$.

We show that the coefficent rings $(I:I)$ and $(J:J)$ are the same. For $x \in (J:J)$ there exists $G \in M_{n}(\mathbb{Z})$ such that $x v=G v$. Then
\[x u=x \sum_{i=1}^{k}a_{i}b_{i} u=x\sum_{i=1}^{k}a_{i}W_{1 i}v=\sum_{i=1}^{k}a_{i}W_{1i}G v=\sum_{i=1}^{k}W_{1i}G M_{i1}u.\]
This implies $x \in (I:I)$. The other inclusion is proved in a similar way. Thus $(I:I)=(J:J)$, and we set $R$ to be this common order.

We verify that $\sum_{i=1}^{k}a_{i}R = (I:J)$ and $\sum_{i=1}^{k}b_{i}R= (J:I)$. For $z \in (I:J)$ there is $Y\in M_{n}(\mathbb{Z})$ such that $z v=Y u$. Then for every $i=1,\dots,k$, we have
\[z b_{i} u=z W_{1i}v=W_{1i}Y u.\]
Hence $z b_{i} \in (I:I) = R$ for every $i=1,\dots,k$. This implies that 
\[ z = z\sum_{i=1}^k a_ib_i = \sum_{i=1}^{k} a_{i} (zb_i) \in \sum_{i=1}^{k}a_{i}R,\]
and hence that $(I:J)\subset \sum_{i=1}^k a_i R$. For the other inclusion, we have for $r_i\in R$, $i=1,\dots,k$, that
\[ (a_1r_1+\cdots+a_kr_k)J \subset a_1J + \cdots + a_kJ \subset I\]
because $r_i\in(J:J)$ and $a_i\in (I:J)$. Thus $\sum_{i=1}^{k}a_{i}R =(I:J)$. The other equality $\sum_{i=1}^{k}b_{i}R= (J:I)$ is proved in a similar way. 

We show that $I$ and $J$ are weakly equivalent. We have $(J:I)I\subset J$ and $(I:J)J\subset I$ by the definitions of $(J:I)$ and $(I:J)$ respectively. Then
\[ (I:J)(J:I)I \subset (I:J)J \subset I.\]
This implies that $(I:J)(J:I) \subset (I:I)=R$. But as $\sum_{i=1}^{k}a_{i}b_{i} =1$, we have the equality $(I:J)(J:I)=R$. Set $X=(J:I)$ and $Y=(I:J)$. Then $XY=R$. Since
\[ J = RJ = (J:I)(I:J) J \subset (J:I)I\subset J\]
implies $(J:I) I= J$, then $IX = I(J:I) = J$. Since
\[ I = RI = (I:J)(J:I) I \subset (I:J)J \subset I\]
implies $(I:J)J = I$, then $JY = J(I:J) = I$. Therefore the ideals $I$ and $J$ are weakly equivalent.

The same conclusion of $I$ and $J$ being weakly equivalent follows from a similar argument if we assume instead a  conjugacy of $\bigoplus_{i=1}^k A$ and $\bigoplus_{i=1}^k B_i$ with $B_1=B$, for some $k\geq 2$. Thus we have proved

\begin{proposition}\label{impliesweak} Suppose $I$ and $J$ are ideals associated with $A,B\in GL_n(\Z)$ with the same irreducible characteristic polynomial $f(t)$. If, for some $k\geq 2$, there exists a conjugacy between $\bigoplus_{i=1}^k B$ and $\bigoplus A_i$ with $A_1=A$, or there exists a conjugacy between $\bigoplus_{i=1}^k A$ and $\bigoplus_{i=1}^k B_i$ with $B_1=B$, then $I$ and $J$ are weakly equivalent.
\end{proposition}

\begin{remark}{\rm The proof of Proposition \ref{impliesweak} shows that either one of the two conjugacies involved in a $k$-block conjugacy between $A$ and $B$ implies weak equivalence of the associated ideals. This, by Propositions \ref{kimpliesk+1} and \ref{weakequivalence}, has the consequence that the existence of a conjugacy between $\bigoplus_{i=1}^{k} B$ and $\bigoplus_{i=1}^{k} A_{i}$ for some $A=A_1, A_2, \dots, A_k$, implies the existence of a conjugacy between $\bigoplus_{i=1}^{k} A$ and $\bigoplus_{i=1}^{k} B_{i}$ for some $B=B_{1},B_{2}, \dots, B_{k}$, and vice versa.
}\end{remark}

\begin{remark}\label{samewe}{\rm  The proof of Proposition \ref{impliesweak} also shows, by replacing $A$ with $A_i$ or $B$ with $B_i$, that all the ideals associated with  $A=A_1,A_2,\dots,A_k$ and $B=B_1,B_2,\dots,B_k$ present in a $k$-block conjugacy are all weakly equivalent.
}\end{remark}

As a consequence of Propositions \ref{kimpliesk+1}, \ref{weakequivalence}, and \ref{impliesweak}, we have the following characterizations of weak equivalence of ideals.

\begin{theorem}\label{2blockequivalence}
Suppose $I$ and $J$ are ideals associated with $A, B \in GL_{n}(\Z)$ having the same irreducible characteristic polynomial $f(t)$. Then the following are equivalent.
\begin{enumerate}
\item[(a)] $I$ and $J$ are weakly equivalent.
\item[(b)] $A$ and $B$ are $2$-block conjugate.
\item[(c)] $A$ and $B$ are $k$-block conjugate for all $k\geq 2$.
\item[(d)] $A$ and $B$ are block conjugate.
\end{enumerate}
\end{theorem}

Consequently, the relations of $2$-block conjugacy and $k$-block conjugacy are equivalent for all $k\geq 3$. This means we need only be concerned with $2$-block conjugacy when comparing two automorphisms. Because $2$-block conjugacy and weak equivalence of ideals are equivalent by Theorem \ref{2blockequivalence}, and because weak equivalence of ideals is an equivalence relation, we have

\begin{corollary} The relation of $2$-block conjugacy on $GL_n(\Z)$ matrices with irreducible characteristic polynomial $f(t)$ is an equivalence relation.
\end{corollary}

Another consequence of Theorem \ref{2blockequivalence} is the following result, which is analogous to the Latimer-MacDuffee-Taussky Theorem.

\begin{corollary} There is a bijection between the $2$-block conjugacy classes of $GL_n(\Z)$ matrices with irreducible characteristic polynomial $f(t)$ and the weak equivalence classes of ideals of $\Z[\beta]$.
\end{corollary}

\begin{remark}{\rm 
For a weak equivalence class of ideals, with ring of coefficients $R$, there is a bijection between the arithmetic equivalence classes contained in it and the elements of the  class group of $R$, i.e., the arithmetic equivalence classes of invertible $R$-ideals. So, up to a point, the problem of conjugacy is reduced to the conjugacy of automorphisms associated to invertible ideals plus $2$-block conjugacy.
}\end{remark}

\section{Invariant tori for $B \oplus B$ and induced actions} A two-block conjugacy of two irreducible toral automorphisms presents a new platform on which to compare the dynamics of the automorphisms.  We will focus on the dynamical consequences of a conjugacy $(B\oplus B) M = M(A\oplus A^\prime)$, with similar results holding for the other conjugacy $(A\oplus A)N = N(B\oplus B^\prime)$ of $2$-block conjugacy.

A pair of complementary $(B\oplus B)$-invariant linearly embedded $n$-dimensional tori of ${\mathbb T}^{2n}$ are determined by the conjugating matrix $M\in GL_{2n}(\Z)$ in $(B\oplus B)M = M(A\oplus A^\prime)$. These two subtori are given by the  images of ${\mathbb T}^n$ by the $n\times 2n$ block columns
\[ \begin{pmatrix} M_1 \\ M_3 \end{pmatrix}, \ \begin{pmatrix} M_2 \\ M_4\end{pmatrix},\]
of the conjugating $M$. The dynamics of $B\oplus B$ on these invariant $n$-dimensional subtori are precisely the dynamics of $A$ and $A^\prime$ respectively.

A linearly embedded $n$-dimensional torus of ${\mathbb T}^{2n}$ is the image of many different linear embeddings. For a linear embedding $\begin{pmatrix} M_1 \\ M_3\end{pmatrix}$ and any $P \in GL_{n}(\Z)$, the linear embedding 
$\begin{pmatrix} M_1 \\ M_3\end{pmatrix} P$
has the same image as that of  $\begin{pmatrix} M_1 \\ M_3\end{pmatrix}$. In fact, all linear embeddings of $n$-dimensional tori in ${\mathbb T}^{2n}$ are of this form. To show this, suppose that $\begin{pmatrix} U_1 \\ U_3 \end{pmatrix}$ has the same image as that of $\begin{pmatrix} M_1 \\ M_3\end{pmatrix}$. By the Hermite Normal Form, there exist $M_2,M_4\in M_n(\Z)$ such that
\[ M = \begin{pmatrix} M_1 & M_2 \\ M_3 & M_4\end{pmatrix} \in GL_{2n}(\Z).\]
Set $W=M^{-1}$. Since the two linear embeddings have the same image, for each $\theta\in{\mathbb T}^n$ there is $\Theta\in{\mathbb T}^n$ such that
\[ W\begin{pmatrix} U_1 \\ U_3 \end{pmatrix}(\theta) = \begin{pmatrix}  \Theta \\ 0\end{pmatrix}.\]
The mapping $\theta \to \Theta$ is an automorphism of ${\mathbb T}^n$, so there is $P\in GL_n(\Z)$ such that
\[ W \begin{pmatrix} U_1 \\ U_3 \end{pmatrix} = \begin{pmatrix} P \\ 0 \end{pmatrix}.\]
Hence we have that
\[ \begin{pmatrix} U_1 \\ U_3 \end{pmatrix} = M \begin{pmatrix} P \\ 0 \end{pmatrix} = \begin{pmatrix} M_1 \\ M_3 \end{pmatrix} P.\]

We can identify the linearly embedded $n$-dimensional tori of ${\mathbb T}^{2n}$ which are $(B\oplus B)$-invariant in terms of a semiconjugacy. A linearly embedded $n$-dimensional torus $\begin{pmatrix} M _1 \\ M_3 \end{pmatrix} {\mathbb T}^n$ is $(B \oplus B)$-invariant when there exists an $A \in GL_{n}(\Z)$ such that $A$ is semiconjugate to $B\oplus B$ by the linear embedding $\begin{pmatrix} M_1 \\ M_3\end{pmatrix}$. That is, we have 
\[ (B\oplus B) \begin{pmatrix} M_1 \\ M_3\end{pmatrix} = \begin{pmatrix} M_1 \\ M_3\end{pmatrix} A.\]
In this case, we say that $B\oplus B$ induces $A$ on the $(B\oplus B)$-invariant $n$-dimensional subtorus $\begin{pmatrix} M_1 \\ M_3\end{pmatrix} {\mathbb T}^n$. For a different embedding $\begin{pmatrix} U_1 \\ U_3\end{pmatrix}$ with the same image as that of $\begin{pmatrix} M_1 \\ M_3\end{pmatrix}$, the action of $B\oplus B$ induces the conjugate $P^{-1}AP$ on the $(B\oplus B)$-invariant subtorus $\begin{pmatrix} U_1 \\ U_3\end{pmatrix} {\mathbb T}^n$.

Each linearly embedded $n$-dimensional torus of ${\mathbb T}^{2n}$ has a complementary linearly embedded $n$-dimensional torus. For a linear embedding $\begin{pmatrix} M_1 \\ M_3\end{pmatrix}$, there is by the Hermite Normal Form matrices $M_2,M_4\in M_n(\Z)$ such that
\[ \begin{pmatrix} M_1 & M_2 \\ M_3 & M_4 \end{pmatrix} \in GL_{2n}(\Z).\]
The linearly embedded $n$-dimensional torus $\begin{pmatrix} M_2 \\ M_4\end{pmatrix} {\mathbb T}^n$ is a complement of the linearly embedded $n$-dimensional torus $\begin{pmatrix} M_1 \\ M_3\end{pmatrix} {\mathbb T}^n$. Even if $\begin{pmatrix} M_1 \\ M_3\end{pmatrix} {\mathbb T}^n$ is $(B\oplus B)$-invariant, it might not be that a complementary $\begin{pmatrix} M_2 \\ M_4\end{pmatrix} {\mathbb T}^n$ is $(B\oplus B)$-invariant.

\begin{definition} A $(B \oplus B)$-invariant linearly embedded $n$-dimensional torus in ${\mathbb T}^{2n}$ is invariantly complemented if there exist $M_{2}, M_{4} \in M_n(\Z)$ and $A,A^\prime \in GL_n(\Z)$ such that
\[ M = \begin{pmatrix} M_1 & M_2 \\ M_3 & M_4\end{pmatrix} \in GL_{2n}(\Z)\]
and
\[(B \oplus B)M=M (A \oplus A^\prime).\]
\end{definition}

Certainly, having a conjugacy $(B\oplus B)M =  M(A\oplus A^\prime)$ for $M\in GL_{2n}(\Z)$ and $A^\prime\in GL_n(\Z)$ (which follows from $A$ and $B$ being $2$-block conjugate) implies that the complement $\begin{pmatrix} M_2 \\ M_4\end{pmatrix} {\mathbb T}^n$ of the $(B\oplus B)$-invariant $\begin{pmatrix} M_1 \\ M_3\end{pmatrix} {\mathbb T}^n$ is also $(B\oplus B)$-invariant. However having a conjugacy
\[ (B\oplus B) M = M \hat A = M \begin{pmatrix} A & S \\ 0 & A^\prime \end{pmatrix}\]
with $S\ne 0$, implies that the complement $\begin{pmatrix} M_2 \\ M_4\end{pmatrix} {\mathbb T}^n$ of the $(B\oplus B)$-invariant $\begin{pmatrix} M_1 \\ M_3\end{pmatrix} {\mathbb T}^n$ is not $(B\oplus B)$-invariant. There are $B$ for which a $(B\oplus B)$-invariant linearly embedded $n$-dimensional torus is not invariantly complemented, as illustrated next.

\begin{example}{\rm Taken from \cite{MS2}, the $GL_3(\Z)$ matrices
\[ A = \begin{pmatrix} -1 & 2 & 0 \\ -1 & 1 & 1 \\ -8 & -6 & -23\end{pmatrix},\ B=\begin{pmatrix} 0 & 1 & 0 \\ -1 & 0 & 2 \\ -11 & -3 & 23\end{pmatrix},\]
both have the hyperbolic and irreducible $f(t) = t^3 - 23t^2 + 7t - 1$ as their characteristic polynomial. Associated to $A$ and $B$ are the ideals $I$ and $J$ with $\Z$-bases $(2,\beta+1,\beta^2+1)$ and $(1,\beta,(\beta^2+1)/2)$ respectively where $f(\beta) = 0$. The ideal J is actually an order $R$, lying above $\Z[\beta]$, and hence $J$ is invertible. The ideal $I$ has $R$ as its ring of coefficients, but is not invertible. Thus $I$ and $J$ are not weakly equivalent, and so $A$ and $B$ are not $2$-block conjugate by Theorem \ref{2blockequivalence}. But there is a semiconjugacy from $A$ to $B\oplus B$ by the linear embedding
\[ \begin{pmatrix} M_1 \\ M_3 \end{pmatrix} = \begin{pmatrix} 1 & 0 & 0 \\ -1 & 2 & 0 \\ 0 & 0 & 1 \\ -1 & 1 & 0 \\ 0  & -1 & 1 \\ -4 & -3 & 11 \end{pmatrix},\]
that is, we have
\[ (B\oplus B) \begin{pmatrix} M_1 \\ M_3\end{pmatrix} =  \begin{pmatrix} M_1 \\ M_3\end{pmatrix} A.\]
By Proposition \ref{nkconjugacy}, there exist $M\in GL_{2n}(\Z)$, $A^\prime\in GL_n(\Z)$, and $S\in M_n(\Z)$ such that
\[ M = \begin{pmatrix} M_1 & M_2 \\ M_3 & M_4\end{pmatrix} {\rm\ and\ } (B\oplus B)M = M \hat A = M \begin{pmatrix} A & S \\ 0 & A^\prime\end{pmatrix}.\]
The matrix $S$ can never be $0$ because if it were for some choice of $M_2$ and $M_4$, then by Proposition \ref{impliesweak}, the ideals $I$ and $J$ would be weakly equivalent. Thus the linearly embedded $n$-dimensional torus $\begin{pmatrix} M_1 \\ M_3\end{pmatrix} {\mathbb T}^n$ is not invariantly complemented.
}\end{example}

\subsection{Equivalence of embeddings and conjugacy} We present a characterization of when two irreducible toral automorphisms $A$ and $B$ that are $2$-block conjugate are actually conjugate. The setting for this characterization is the set of $(B\oplus B)$-invariant linearly embedded $n$-dimensional tori of ${\mathbb T}^{2n}$ that are invariantly complemented. We can identify this set with
\[ {\mathcal I} = \{M \in GL_{2 n}(\Z):M^{-1}(B \oplus B)M \, \mbox{ is block diagonal}\}.\] 
There is a right action of $GL_n(\Z)\times GL_n(\Z)$ on ${\mathcal I}$ given by $(P,Q)\cdot M = M(P\oplus Q)$ where $P,Q\in GL_n(\Z)$. This action accounts for all of the  different embeddings with the same images.

We consider the automorphisms of ${\mathbb T}^{2n}$ that preserve the set of $(B \oplus B)$-invariant linearly embedded $n$-dimensional tori that are invariantly complemented. Let $\mathcal{E}$ be the subset of $\xi\in GL_{2 n}(\Z)$ such that $\xi M\in {\mathcal I}$ and $\xi^{-1}M\in {\mathcal I}$ for all $M\in{\mathcal I}$. Note that ${\mathcal E}$ is a subset of ${\mathcal I}$. Each $M\in{\mathcal I}$ has associated to it $A,D\in GL_n(\Z)$ (that depend on $M$) such that $(B\oplus B)M = M(A\oplus D)$. An element $\xi \in GL_{2n}(\Z)$ belongs to ${\mathcal E}$ if for all $M\in {\mathcal I}$ there exist $A_1,D_1, A_2,D_2 \in GL_n(\Z)$ (depending on $M$ and $\xi$) such that
\[ (B\oplus B) \xi M = \xi M (A_1\oplus D_1), \ (B\oplus B)\xi ^{-1}M = \xi^{-1}M(A_2\oplus D_2).\]
The set ${\mathcal E}$ is a group because it contains inverses by definition, and for $\xi,\eta\in{\mathcal E}$ we have $(\xi\eta)M = \xi(\eta M)\in{\mathcal I}$ and $(\xi^{-1}\eta^{-1})M = \xi^{-1}(\eta^{-1} M)\in {\mathcal I}$ for all $M\in {\mathcal I}$.

The group $\mathcal{E}$ contains as a subgroup the centralizer
\[ {\mathcal C} = \{ U\in GL_{2n}(\Z): (B\oplus B)U = U(B\oplus B)\}.\]
Each $U \in \mathcal{C}$ maps every $(B\oplus B)$-invariant linearly embedding $n$-dimensional torus and its invariant complement to another such pair with the same induced actions because for $M\in {\mathcal I}$ we have
\[ (B\oplus B)U M = U (B\oplus B)M =  U M(A\oplus D).\]
Combining the left action of ${\mathcal C}$ on ${\mathcal I}$ with the right action of $GL_n(\Z)\times GL_n(\Z)$ on ${\mathcal I}$ gives the action of the group ${\mathcal C} \times (GL_n(\Z)\times GL_n(\Z))$ on ${\mathcal I}$ defined by
\[ M \mapsto UM(P\oplus Q).\]

\begin{proposition}
The set ${\mathcal I}$ of $(B \oplus B)$-invariant and invariantly complemented linearly embedded $n$-dimensional tori is partitioned by the conjugacy classes of the induced actions. The group action of $\mathcal{C}\times (GL_n(\Z)\times GL_n(\Z))$ on ${\mathcal I}$ preserves this partition and acts transitively in each element of the partition.
\end{proposition}

\begin{proof} Suppose for $M,V\in{\mathcal I}$ with first block columns $\begin{pmatrix}M_1 \\ M_3\end{pmatrix}$ and $\begin{pmatrix} V_1 \\ V_3\end{pmatrix}$ that the induced actions of $B\oplus B$ on $\begin{pmatrix}M_1 \\ M_3\end{pmatrix} {\mathbb T}^n$ and $\begin{pmatrix} V_1 \\ V_3\end{pmatrix} {\mathbb T}^n$ are $A$ and $C$ respectively. If $A$ and $C$ are not conjugate, then there is no $(U,P,Q) \in {\mathcal C} \times (GL_n(\Z)\times GL_n(\Z))$ such that $UM(P\oplus Q) = V$ because $U$ does not change the induced action, and $P$ conjugates $A$. If $A$ and $C$ are conjugate, then we can replace $V$ with $I_{2n}V(P\oplus I_n)$ for some $P\in GL_n(\Z)$ so that $B\oplus B$ induces $A$ on $\begin{pmatrix} V_1 \\ V_3\end{pmatrix} {\mathbb T}^n$. Thus, we suppose that
\[ (B\oplus B) \begin{pmatrix} M_1 \\ M_3 \end{pmatrix} = \begin{pmatrix} M_1 \\ M_3 \end{pmatrix}  A,\]
and
\[ (B\oplus B) \begin{pmatrix} V_1 \\ V_3 \end{pmatrix} = \begin{pmatrix} V_1 \\ V_3 \end{pmatrix} A.\]

The existence of $U \in \mathcal{C}$ for which $U \begin{pmatrix} M_1 \\ M_3\end{pmatrix} = \begin{pmatrix} V_1 \\ V_3 \end{pmatrix}$ is implied by the existence of $(B\oplus B)$-invariant complements for the two subtori such that the induced actions of $B\oplus B$ on the complements are the same. Indeed, if there exist $M_2,M_4,V_2,V_4 \in M_n(\Z)$ and $D \in GL_n(\Z)$ such that 
\[ M = \begin{pmatrix} M_1 & M_2 \\ M_3 & M_4\end{pmatrix} \in GL_{2n}(\Z), \ V=\begin{pmatrix} V_1 & V_2 \\ V_3 & V_4\end{pmatrix} \in GL_{2n}(\Z),\]
and
\[ (B\oplus B)M = M(A\oplus D),\ (B\oplus B)V = V(A\oplus D)\]
(these last two equations meaning that $M,V\in{\mathcal I}$), then $U=VM^{-1} \in {\mathcal C}$ and satisfies $U \begin{pmatrix} M_1 \\ M_3\end{pmatrix} = \begin{pmatrix} V_1 \\ V_3 \end{pmatrix}$.

Conversely, the existence of $U\in{\mathcal C}$ for which $U \begin{pmatrix} M_1 \\ M_3\end{pmatrix} = \begin{pmatrix} V_1 \\ V_3 \end{pmatrix}$ implies the existence of $(B\oplus B)$-invariant complements for the two subtori such that the induced actions of $B\oplus B$ on the complements are the same. Take any $M_2,M_4 \in M_n(\Z)$ such that
\[ M = \begin{pmatrix} M_1 & M_2 \\ M_3 & M_4\end{pmatrix} \in GL_{2n}(\Z) \cap {\mathcal I}.\]
Then $M$ satisfies $(B\oplus B)M = M(A\oplus D)$ for some $D \in GL_n(\Z)$. The matrix $V=UM\in GL_{2n}(\Z)$, its first block column is $\begin{pmatrix} V_1 \\ V_3\end{pmatrix}$, and because
\[ (B\oplus B)V = (B\oplus B)UM = U(B\oplus B)M = UM(A\oplus D) = V(A\oplus D)\]
we have $V\in {\mathcal I}$.
  
It remains to show the existence of $(B\oplus B)$-invariant complements for the two subtori such that the induced actions of $B\oplus B$ on the complements are the same. Since the two $(B\oplus B)$-invariant subtori $\begin{pmatrix} M_1 \\ M_3\end{pmatrix} {\mathbb T}^n$ and $\begin{pmatrix} V_1 \\ V_3\end{pmatrix} {\mathbb T}^n$ are invariantly complemented, there are $M_2,M_4,V_2,V_4\in M_n(\Z)$ and $D,D^\prime \in GL_n(\Z)$ such that
\[ M = \begin{pmatrix} M_1 & M_2 \\ M_3 & M_4\end{pmatrix} \in GL_n(\Z),\ V = \begin{pmatrix} V_1 & V_2 \\ V_3 & V_4\end{pmatrix} \in GL_n(\Z),\]
and
\[(B \oplus B)M = M(A \oplus D),\ (B \oplus B)V=V(A \oplus D').\]
These conjugacies imply by Remark \ref{samewe} the weak equivalence of the ideals $J$ and $F$ associated to $B$ and $D$ respectively, and the weak equivalence of the ideals $J$ and $F^\prime$ associated to $B$ and $D^\prime$ respectively. Mimicking the proof of Proposition \ref{impliesweak}, the ideal $G=(F:J)$ is invertible and satisfies $JG=F$, and the ideal $G^\prime = (F^\prime:J)$ is invertible and satisfies $JG^\prime = F^\prime$. With $v=(v_1,\dots,v_n)^t$ a $\Z$-basis for $J$ and $w=(w_1,\dots,w_n)^t$ a $\Z$-basis for $G$, there are $c_1,c_2\in (F:J)$ such that $M_2w=c_1v$ and $M_4 w = c_2 v$ because $BM_2=M_2D$ and $BM_4=M_4D$. The constants $c_1$ and $c_2$ generate $G$, i.e., $G=c_1R+c_2R$, where $R$ is the common ring of coefficients of the weakly equivalent ideals. With $u=(u_1,\dots,u_n)^t$ a $\Z$-basis for the ideal $I$ associated to $A$, the ideal $X=(J:I)$ does not depend on $M$ or
\[ W = M^{-1} = \begin{pmatrix} W_1 & W_2 \\ W_3 & W_4\end{pmatrix},\]
but the constants $b_1,b_2\in(J:I)$ in $X=b_1R+b_2R$ determined by $W_1 v = b_1u$ and $W_2 v = b_2 u$ do depend on $W$. Because $WM=I_{2n}$, we have $W_1M_2+W_2M_4=0$, and so
\[ 0 = (W_1M_2+W_2M_4)w = (b_1c_1+b_2c_2)u,\]
implying that $b_1c_1+b_2c_2=0$. Then
\begin{align*} c_1X 
& = c_1(b_1R+b_2R) \\
& = c_1b_1 R + c_1b_2 R \\
& = -c_2b_2R + c_1b_2R \\
& = b_2(-c_2R+c_1R) = b_2G.
\end{align*}
Thus $X$ and $G$ are arithmetically equivalent ideals. In a similar manner, the ideals $X$ and $G^\prime$ are arithmetically equivalent. Then $F=JG$ and $F^\prime = JG^\prime$ are arithmetically equivalent ideals. Hence $D$ and $D^\prime$ are conjugate, so there is $Q\in GL_n(\Z)$ such that $D = Q D^\prime Q^{-1}$. Multiplying $V$ on the right by $I_n\oplus Q$ gives a new $\tilde V\in GL_n(\Z)$ such that $(B\oplus B) \tilde V = \tilde V(A\oplus D)$.
\end{proof}

\begin{remark}{\rm  An argument similar to the one above that shows the arithmetic equivalence of the invertible ideals $X=(J:I)$ and $G=(F:J)$, also shows the arithmetic equivalence of the invertible ideals $Y=(I:J)$ and $H=(J:F)$, where $XY=R$ and $GH=R$.
}\end{remark}

One element of the partition of ${\mathcal I}$ is the equivalence class represented by the conjugacy class of $B$. This is because the torus $\begin{pmatrix} I_n \\ 0 \end{pmatrix} {\mathbb T}^n$ is $(B\oplus B)$-invariant with induced action $B$, and it has a complement $\begin{pmatrix} 0 \\ I_n\end{pmatrix} {\mathbb T}^n$ that is $(B\oplus B)$-invariant. If $(B\oplus B)M = M(A\oplus D)$ for some $A,D\in GL_n(\Z)$ and
\[ M = \begin{pmatrix} M_1 & M_2 \\ M_3 & M_4\end{pmatrix} \in GL_{2n}(\Z),\]
then the induced action of $B\oplus B$ on $\begin{pmatrix} M_1 \\ M_3\end{pmatrix} {\mathbb T}^n$ is $A$, and another element of the partition of ${\mathcal I}$ is the equivalence class represented by the conjugacy class of $A$. The two equivalence classes represented by the conjugacy classes of $B$ and $A$ are either disjoint, or the same. If disjoint, then $A$ and $B$ are not conjugate. If the same, then by the transitivity of ${\mathcal C}\times(GL_n(\Z)\times GL_n(\Z))$ on each equivalence class of ${\mathcal I}$, we have that $A$ and $B$ are conjugate. Thus we have proved the following which gives a characterization of when $2$-block conjugate $A$ and $B$ are $1$-block conjugate.

\begin{corollary} For $A,B\in GL_n(\Z)$, suppose there exist $M\in GL_{2n}(\Z)$ and $D \in GL_n(\Z)$ such that $(B\oplus B)M = M(A\oplus D)$ and let $\begin{pmatrix} M_1 \\ M_3\end{pmatrix}$ be the first block column of $M$. Then $A$ and $B$ are conjugate if and only if there exists $U \in \mathcal{C}$ mapping $\begin{pmatrix} M_1 \\ M_3\end{pmatrix} {\mathbb T}^n$ to $\begin{pmatrix} I_n \\ 0\end{pmatrix} {\mathbb T}^n$.
\end{corollary}

\begin{remark}{\rm  Each of the $n\times n$ blocks in
\[ U = \begin{pmatrix} U_1 & U_2 \\ U_3 & U_4\end{pmatrix} \in {\mathcal C}\]
commutes with $B$. For the eigenvector $v$ of $B$ corresponding to the eigenvalue $\beta$, we have for each $i=1,2,3,4$ that $U_i v = \alpha_i v$ for some $\alpha_i \in R = (J:J)$. More precisely, we have that $\alpha_i = p(\beta)$ where $U_i = p(B)$ because $B$ is irreducible (see \cite{BR}). Thus we have $U_iU_j = U_jU_i$ for all $i\ne j$ which implies that
\[ {\rm det}(U) = {\rm det}(U_1U_4 - U_2U_3)\]
(see \cite{Sil}). Furthermore we have
\[ (U_1U_4 - U_2U_3)v = (\alpha_1\alpha_4 - \alpha_2\alpha_3)v,\]
so that $\alpha_1\alpha_4-\alpha_2\alpha_3$ is a unit in $R$ if and only if ${\rm det}(U) = \pm 1$. We may thus identify ${\mathcal C}$ with $GL_2(R)$.
}\end{remark}

\subsection{Relationship with a Galois Group} We detail additional properties of the elements of ${\mathcal E}$ that lead to the relationship of ${\mathcal E}$ with the Galois group $Gal(K/{\mathbb Q})$. Recall that to each $M\in{\mathcal I}$ there is associated $A,D\in GL_n(\Z)$, depending on $M$, such that
$(B\oplus B)M = M(A\oplus D)$, and to each $\xi \in{\mathcal I}$ there is associated $A_1,D_1\in GL_n(\Z)$, depending on $M$ and $\xi$, such that $(B\oplus B)\xi M = \xi M(A_1\oplus D_1)$. By taking $M=I_{2n}$, there is for each $\xi\in {\mathcal E}$ the existence of $A_\xi,D_\xi\in GL_n(\Z)$ such that $(B\oplus B)\xi = \xi (A_\xi \oplus D_\xi)$. For an arbitrary $M\in{\mathcal I}$ we have
\[ \xi(A_\xi\oplus D_\xi)M = (B\oplus B)\xi M = \xi M(A_1\oplus D_1).\]
This gives the dependence of $A_1$ and $D_1$ on $A_\xi$ and $D_\xi$ and arbitrary $M\in{\mathcal I}$ by
\[ A_1\oplus D_1 = M^{-1}(A_\xi \oplus D_\xi)M\]
because $\xi$ and $M$ are invertible. For the matrix
\[ M = \begin{pmatrix} I & I \\ I & 2I\end{pmatrix}\]
which belongs to ${\mathcal C}$, and hence to ${\mathcal I}$, we have
\[ \begin{pmatrix} A_1 & D_1 \\ A_1 & 2D_1\end{pmatrix} = M(A_1\oplus D_1) = (A_\xi \oplus D_\xi)M = \begin{pmatrix} A_\xi & A_\xi \\ D_\xi & 2D_\xi\end{pmatrix}.\]
This implies that $A_\xi = D_\xi$, and hence the dependence of $A_1$ and $D_1$ on $\xi$ for arbitrary $M\in{\mathcal I}$ becomes
\[ A_1\oplus D_1 = M^{-1}(A_\xi \oplus A_\xi) M.\]
For the matrix
\[ M = \begin{pmatrix} B^2 & 0 \\ B & B^{-1}\end{pmatrix},\]
which belongs to ${\mathcal C}$, and hence to ${\mathcal I}$, we have
\[ \begin{pmatrix} B^2 A_1 & 0 \\ BA_1 & B^{-1}D_1\end{pmatrix} = M(A_1\oplus D_1) = (A_\xi \oplus A_\xi) M = \begin{pmatrix} A_\xi B^2 & 0 \\ A_\xi B & A_\xi B^{-1}\end{pmatrix}.\]
These implies that
\[ A_\xi B = B A_1 = B^{-1}(B^2 A_1) = B^{-1}(A_\xi B^2),\]
hence that $BA_\xi = A_\xi B$. By the irreducibility of $B$ we have that $A_\xi = p_{\xi}(B)$ for a polynomial $p_{\xi}(t)$ with rational coefficients, and so $ A_{\xi} v = p_{\xi }(\beta)v$. Since $(B\oplus B)\xi = \xi(A_\xi \oplus A_\xi)$, the matrix $A_{\xi}$ is similar (over $\Q$) to $B$, and so the quantity  $p_{\xi}(\beta)$ is also an eigenvalue of $B$. Hence the polynomial $p_{\xi}(t) \in \Q[t]$ represents the element $p_{\xi}(\beta)$ of $K=\Q(\beta)$, as well as a $\Q$-automorphism of $K=\Q[t]/(f(t))$, i.e., an element $\phi_{\xi}$ of $Gal(K/\Q)$ defined by
\[ \phi_{\xi}(\alpha)=\alpha \circ p_{\xi},\]
where we identify the element $\alpha \in K$ with one of its polynomial representations. 

We establish that the map $\xi \mapsto \phi_\xi$ from ${\mathcal E}$ to $Gal(K/\Q)$ is a group antihomomorphism and determine it kernel. An element $\xi \in {\mathcal E}$ satisfies $(B\oplus B)\xi = \xi(A_\xi\oplus A_\xi)$ where $A_\xi = p_\xi(B)$ is similar to $B$ over ${\mathbb Q}$. One of the $n\times n$ blocks $\xi_i$ in
\[ \xi = \begin{pmatrix} \xi_1 & \xi_2 \\ \xi_3 & \xi_4\end{pmatrix}\]
is invertible and so we have the similarity $B\xi_i = \xi_i A_\xi$ for some $i$. This implies that $w=\xi_i v$ is an eigenvector of $B$ for an eigenvalue $\gamma=p_\xi(\beta)$ of $B$ because
\[ Bw = B(\xi_i v) = \xi_i(A_\xi v) = \xi_i (p_\xi(B) v) = \xi_i( p_\xi(\beta) v) = p_\xi(\beta) w.\]
For $\xi,\eta\in{\mathcal E}$, we have
\[ \eta\xi(\xi^{-1}(A_\eta\oplus A_\eta) \xi) = \eta(A_\eta\oplus A_\eta)\xi = (B\oplus B)\eta \xi,\]
and so
\[ A_{\eta\xi}\oplus A_{\eta\xi} = \xi^{-1}(A_\eta\oplus A_\eta)\xi.\]
From this we get $A_{\eta\xi} = \xi_i^{-1} A_\eta \xi_i$ for the same choice of $i$ as above. Then we have
\[ A_{\eta\xi}v = p_{\eta\xi}(B) v = p_{\eta\xi}(\beta) v\]
and
\[ A_{\eta\xi}v = \xi_1^{-1} A_\eta \xi_i v = \xi_i^{-1} p_\eta(B) w = \xi_i^{-1} p_{\eta}(\gamma)w = p_{\eta}(\gamma) v.\]
These imply that
\[ p_{\eta\xi}(\beta) = p_{\eta}(\gamma) = p_{\eta}(p_\xi(\beta)).\]
Then
\[ \big[ p_{\eta\xi}(B) - p_{\eta}(p_\xi(B))\big] v = \big[ p_{\eta\xi}(\beta) - p_{\eta}(p_\xi(\beta))]v = 0,\]
so that
\[ p_{\eta\xi}(B) = p_{\eta}(p_\xi(B)) = (p_{\eta}\circ p_{\xi}) (B).\]
A straightforward argument shows that the polynomials $p_{\eta\xi}$ and $p_\eta\circ p_\xi$ represent the same element of $\Q[t] / (f(t))$. Consequentially, modulo the ideal $(f(t))$, we have
\[ \phi_{\eta\xi}(\alpha) = \alpha \circ p_{\eta\xi} = \alpha \circ p_\eta \circ p_\xi = \phi_\xi (\alpha \circ p_\eta) = (\phi_\xi\circ \phi_\eta)(\alpha).\]
Finally, for $\xi \in {\mathcal E}$, any one of the four conditions $p_\xi(t) = t$, $\phi_{\xi}={\rm id}$, $A_{\xi}=B$, and $\xi \in {\mathcal C}$ implies the other three. Thus we have proved

\begin{proposition}
The map $\pi:\mathcal{E}\rightarrow Gal(K/\Q)$ given by $\pi(\xi)=\phi_{\xi}$ is a group antihomomorphism with kernel $\mathcal{C}$.
\end{proposition}

We do not know precisely what the image of $\pi$ is for a given $B$. But we do know of one sufficient condition by which ${\mathcal C}$ is a proper subgroup of ${\mathcal E}$, and hence the image of $\pi$ is not trivial.

\begin{proposition} If there exists $\xi \in GL_{2n}(\Z)$ such that $(B\oplus B)\xi = \xi(B^{-1}\oplus B^{-1})$, then $\xi \in {\mathcal E}\setminus{\mathcal C}$.
\end{proposition}

\begin{proof} Suppose $(B\oplus B)\xi = \xi(B^{-1}\oplus B^{-1})$. Applying inverses to both sides gives
\[ \xi^{-1}(B^{-1}\oplus B^{-1}) = (B\oplus B)\xi^{-1}.\]
Then for arbitrary $M\in{\mathcal I}$ we have
\[ (B\oplus B)\xi M = \xi(B^{-1}\oplus B^{-1})M,\ (B\oplus B)\xi^{-1} M = \xi^{-1}(B^{-1}\oplus B^{-1})M.\]
For each $M$ there is $A,D\in GL_n(\Z)$ such that $(B\oplus B)M = M(A\oplus D)$ from which we have
\[ (B^{-1}\oplus B^{-1}) M = M(A^{-1}\oplus D^{-1}).\]
This implies that
\[ (B\oplus B)\xi M = \xi M(A^{-1}\oplus D^{-1}),\ (B\oplus B)\xi^{-1} M = \xi^{-1} M (A^{-1} \oplus D^{-1}).\]
Hence $\xi$ belongs to ${\mathcal E}$. But $\xi$ does not belong to ${\mathcal C}$ because $B^{-1}\ne B$.
\end{proof}

The existence of $\xi \in GL_{2n}(\Z)$ for some $B$ for which $(B\oplus B)\xi = \xi(B^{-1}\oplus B^{-1})$ is not vacuous, as illustrated next.

\begin{example}\label{nontrivial}{\rm Consider the matrix
\[ B = \begin{bmatrix} 3 & 10 \\ 5 & 17\end{bmatrix}\]
with the irreducible hyperbolic characteristic polynomial $f(t) = t^2 - 20t + 1$. The inverse of $B$ is
\[ B^{-1} = \begin{bmatrix} 17 & -10 \\ -5 & 3\end{bmatrix}.\]
The $GL_{4}({\mathbb Z})$ matrix
\[ \xi = \begin{bmatrix} 5 & 0 & 0  & 2 \\ 7 & -5 & -1 & 0 \\ 0 & 24 & 5 & 14 \\ -12 & 0 & 0  &-5\end{bmatrix}\]
satisfies $(B\oplus B) \xi = \xi (B^{-1}\oplus B^{-1})$. Note that $\xi^2 = I_4$. For $K=\Q(\beta)$ where $f(\beta)=0$, the Galois group $Gal(K/\Q)$ is isomorphic to $\Z_2$, and so the antihomorphism $\pi$ is surjective with $\pi(\phi_\xi)$ being the nontrivial element of $Gal(K/\Q)$. 
}\end{example}

\begin{remark}{\rm The matrices $B$ and $B^{-1}$ in Example \ref{nontrivial} are not conjugate (see \cite{ATW}). By setting $A=B^{-1}$, we have another example of $A\oplus A \approx B\oplus B$ with $A$ and $B$ nonconjugate.
}\end{remark}

\section{Non-Block Conjugate Automorphisms} We explore what relationship there may be between ideals $I$ and $J$ associated respectively to non-block conjugate irreducible toral automorphisms $A$ and $B$. By Theorem \ref{2blockequivalence}, the ideals $I$ and $J$ are not weakly equivalent. This does not preclude the existence of a semiconjugacy from $A$ to $\bigoplus_{i=1}^k B$ by a linear embedding for some $k\geq 2$, and hence by Proposition \ref{nkconjugacy}, of a conjugacy between $\bigoplus_{i=1}^k B$ and
\[ \hat A =\begin{pmatrix} A & S \\ 0 & D\end{pmatrix}\]
for some integer $n\times n(k-1)$ matrix $S$ and some $D\in GL_{n(k-1)}(\Z)$. The existence of such a conjugacy may provide the means of constructing an explicit isomorphism between direct summands of ideals involving $I$ and $J$.

\subsection{Semiconjugacy by Linear Embedding} We prove the existence of a semiconjugacy from $A$ to $\bigoplus_{i=1}^k B$ by a linear embedding in a special case. We assume that $B$ is an automorphism $C_R$ corresponding to an order $J=R$ containing $\Z[\beta]$, and that $A$ is an automorphism corresponding to an ideal $I$ whose ring of coefficients is $R$. We set $R=w_{1}\Z+\cdots+w_{n}\Z$ for $w=(w_1,\dots,w_n)^t$ that satisfies $C_{R} w=\beta w$ and $I=u_{1}\Z+\cdots+u_{n}\Z$ where $u=(u_{1},\dots,u_{n})^{t}$ satisfies $A u=\beta u$. As an $R$-module, the ideal $I$ is finitely generated, which means there are finitely many $a_1,\dots,a_k\in R$ such that $I = a_1R + \cdots + a_kR$. Then there exist matrices $X_{i} \in \Lambda(C_{R},A)$, for $i=1,\dots,k$, satisfying
\[X_{i} u=a_{i} w.\]
Since $u_i\in I$ for $i=1,\dots,n$ and $I = a_1R+\cdots+a_kR$, there exist $n$-tuples $z_i\in R^n$, $i=1,\dots,k$, such that $u=\sum_{i=1}^{k}a_{i}z_{i}$. Since $R=w_1\Z + \cdots + w_n\Z$, there exist $Y_{i} \in M_{n}(\Z)$ such that
$z_{i}=Y_{i} w$. As a consequence, we get
\[\sum_{i=1}^{k}Y_{i}X_{i}u=u,\]
which implies that $\sum_{i=1}^{k}Y_{i}X_{i}=I_{n}$. By the Hermite Normal Form there exists  $M \in GL_{k n}(\Z)$ such that partitioning $M$ and $W=M^{-1}$ into $n\times n$ blocks,
\[M=(M_{ij}),\ W=(W_{ij}),\ i,j = 1,\cdots,k,\]
we have
\[M_{i1}=X_{i}\,\, \mbox{  and } W_{1j}=Y_{j}.\]
Then
\[ M^{-1}\left( \bigoplus_{i=1}^k C_R\right) M \begin{pmatrix} u \\ 0 \\ \vdots \\ 0 \end{pmatrix} 
= M^{-1}\begin{pmatrix} C_R X_1 u \\ C_R X_2 u \\ \vdots \\ C_R X_k u \end{pmatrix} 
= M^{-1}\begin{pmatrix} X_1 A u \\ X_2 A u \\ \vdots \\ X_k A u\end{pmatrix}
= \beta \begin{pmatrix} u \\ 0 \\ \vdots \\ 0 \end{pmatrix}.\]
This implies that
\[ M^{-1}\left( \bigoplus_{i=1}^k C_r\right) M = \hat A = \begin{pmatrix} A & S \\ 0 & D\end{pmatrix} \in GL_{kn}(\Z),\]
for an $n\times n(k-1)$ integer matrix $S$, and an $n(k-1)\times n(k-1)$ matrix $D$ belonging to $GL_{n(k-1)}(\Z)$. The conjugacy $M$ may be chosen so that $D$ is, for instance, upper block triangular. However $M$ is chosen, the automorphism $A$ is the action induced by $\bigoplus_{i=1}^{k} C_{R}$ on the linearly embedded $n$-dimensional torus
\[ \begin{pmatrix} X_1 \\ X_2 \\ \vdots \\ X_k\end{pmatrix} {\mathbb T}^n \subset ({\mathbb T}^n)^k.\]
By Proposition \ref{nkconjugacy}, we have proved

\begin{proposition} If the ring of coefficients of an ideal $I$ associated to $A$ is $R$, and the number of generators for $I$ as an $R$-module is $k$, then there is a semiconjugacy from $A$ to $\bigoplus_{i=1}^k C_R$ by a linear embedding.
\end{proposition}

%\begin{remark}The semiconjugacy from $A$ to $\bigoplus_{i=1}^k B$ by linear embedding
%\[ \bigoplus_{i=1}^k C_R \begin{pmatrix} X_1 \\ \vdots \\ X_k\end{pmatrix} =  \begin{pmatrix} X_1 \\ \vdots \\ X_k\end{pmatrix} A\]
%generalizes to any $\alpha\in R$. If we define $A_\alpha \in M_n(\Z)$ by $\alpha u = A_\alpha u$ and $B_\alpha\in M_n(\Z)$ by $\alpha w = B_\alpha w$, then we have the semiconjugacy by linear embedding
%\[ \bigoplus_{i=1}^k B_\alpha \begin{pmatrix} X_1 \\ \vdots \\ X_k\end{pmatrix} =  \begin{pmatrix} X_1 \\ \vdots \\ X_k\end{pmatrix} A_\alpha,\]
%where in general $A_\alpha$ and $\bigoplus_{i=1}^k B_\alpha$ are endomorphisms. If $\alpha = p(\beta)$ for a polynomial $p(t)$ with integer coefficients, then $A_\alpha = p(A)$ and $B_\alpha = p(C_R)$.
%\end{remark}

\subsection{Induced Maps on Direct Sums of Ideals} When the ideal $I$ is invertible, it is generated by one or two elements of $R$, and it is weakly equivalent to $R=(I:I)$. By the implied $2$-block conjugacy of $C_R$ and $A$, there exists an $R$-module isomorphism
\[ \phi: R \oplus R \rightarrow I \oplus (R:I)\]
that was described in Section 4. This confirms the well known result that an ideal of a commutative ring is invertible if and only if it is a projective module, i.e., a direct summand of a free module.

In the general setting, when $I$ is generated by $k$ elements of $R$, there is an isomorphism, not necessarily an $R$-module isomorphism, that relates $I$ and $R$. The $k$ generators $a_1,\dots,a_k$ of $I$ determine a $R$-module epimorphism
\[\psi:\bigoplus_{i=1}^{k} R\rightarrow I {\rm \ defined\ by\ } \psi(x_1,\dots,x_k)=\sum_{i=1}^{ k}a_{i}x_{i}.\]
The $R$-submodule ${\rm ker}(\psi)$ is part of the exact sequence
\[0\rightarrow {\rm ker}(\psi)\rightarrow \bigoplus_{i=1}^{k} R\rightarrow I\rightarrow 0.\]
This exact sequence splits, as a sequence of $R$-modules, if and only if $I$ is an invertible ideal. But this exact sequence always splits as a sequence of abelian groups, giving an isomorphism between $\bigoplus_{i=1}^{k} R$ and $I\oplus {\rm ker}(\psi)$.

We give an explicit construction of an isomorphism between $\bigoplus_{i=1}^k R$ and $I\oplus {\rm ker}(\psi)$ that is induced by the $GL_{nk}(\Z)$ matrix $M$ which conjugates $\bigoplus_{i=1}^k C_R$ and $\hat A$. The conjugating matrix $M$ gives the following commutative diagram.
\[ \begin{CD} \Z^{kn} @>{\bigoplus_{i=1}^k C_R}>> \Z^{kn} \\ @V M VV @VV M V \\ \Z^{kn} @>> \hat A > \Z^{kn}\end{CD}\]
Here, $\Z^{nk} = \bigoplus_{i=1}^k\Z^n$ are $k$-tuples $(m_1,m_2,\dots,m_k)$ of row vectors $m_i$ in $\Z^n$, and the automorphisms act on the right. We identify the two copies of $\Z^{kn}$ in the top row of the commutative diagram with $\bigoplus_{i=1}^k R$ by the isomorphism
\[ (m_1,m_2,\dots,m_k) \to (m_1 w,m_2 w,\dots, m_k w) = (x_1,x_2,\dots,x_k),\]
where each $m_i w$ is the scalar product. The right action of the automorphism $\bigoplus_{i=1}^k C_R$ on $\Z^{kn}$ represents multiplication by $\beta$ on $\bigoplus_{i=1}^k R$ because $m_iC_R w = \beta m_i w$ for each $i=1,\dots,k$. The two copies of $(\Z^n)^k$ in the bottom row of the commutative diagram we think of as $\Z^n\oplus (\Z^n)^{k-1}$, where the first summand $\Z^n$ we identify with $I$ by the isomorphism $s_1\to s_1u$ for $s_1\in \Z^n$. The right action of $A$ on $\Z^n$ represents multiplication by $\beta$ on $I$ because $s_1Au = \beta s_1u$. Extending the isomorphism $s_1\to s_1u$ from $\Z^n$ to $I$ to an isomorphism from $\Z^n \oplus (\Z^n)^{k-1}$ to $I\oplus {\rm ker}(\psi)$ requires obtaining an isomorphic copy of ${\rm ker}(\psi)$.

We explicitly compute ${\rm ker}(\psi)$, showing its dependence on the conjugating matrix $M$. Since $x_i = m_iw$, $X_i u = a_i w$, and $M_{i1} = X_i$, we have
\[ \sum_{i=1}^k a_i x_i = \left(\sum_{i=1}^k m_i M_{i1}\right) u.\]
For $(x_1,x_2,\dots,x_k)$ to belong to ${\rm ker}(\psi)$ requires that $\sum_{i=1}^k m_i M_{i1} = 0$ since the entries of $u$ are a $\Z$-basis for $I$. For arbitrary $s_1,s_2,\dots,s_k \in \Z^n$, we have
\[ \begin{pmatrix} s_1 & s_2 & \dots & s_k \end{pmatrix} M^{-1} \begin{pmatrix} M_{11} \\ M_{21} \\ \vdots \\ M_{k1}\end{pmatrix} = \begin{pmatrix} s_1 \\ 0 \\ \vdots \\ 0\end{pmatrix}.\]
Thus the $(m_1,m_2, \dots, m_k)$ for which $\sum_{i=1}^k m_i M_{i1} = 0$ are given by
\[ m_i = \sum_{j=1}^k s_j W_{ji},\ i=1,2,\dots, k,\]
where $s_1=0$ and $s_2,\dots,s_k\in \Z^n$ are arbitrary. By the identification of $\bigoplus_{i=1}^k R$ with $\Z^{kn}$, the kernel of $\psi$ is then generated over $\Z$ by the last $(k-1)n$ rows of $M^{-1} = W = (W_{ij})$, i.e.,
\[ {\rm ker}(\psi)=\left\{(x_1,\dots,x_k) \in \bigoplus_{i=1}^{k} R: x_{i}=\sum_{j=2}^{k}s_{j}W_{ji}w,\, (s_2,\dots,s_k) \in (\Z^{n})^{k-1}\right\}.\]

Now we can extend the isomorphism $s_1\to s_1 u$ from $\Z^n$ to $I$ to an isomorphism from $\Z^n \oplus(\Z^n)^{k-1}$ to $I\oplus {\rm ker}(\psi)$. For each $(x_1,x_2,\dots,x_k)\in {\rm ker}(\psi)$, the values of $x_2,\dots, x_k$ are determined uniquely by $s_2,\dots,s_k\in \Z^n$, and so the value of $x_1$ is uniquely determined by the values of $x_2,\dots,x_k$. Thus there is an isomorphism from ${\rm ker}(\psi)$ to the $R$-module of elements of the form
\[ \left( \sum_{j=2}^k s_jW_{j2} w, \dots, \sum_{j=2}^k s_j W_{jk}w\right)\]
given by the projection $(x_1,x_2,\dots,x_k)\to(x_2,\dots,x_k)$. By abuse of notation we denote by ${\rm ker}(\psi)$ this isomorphic copy of ${\rm ker}(\psi)$. We identify the two copies of $\Z^{k n}=\Z^{n}\oplus (\Z^{n})^{k-1}$ in the bottom row of the commutative diagram with $I \oplus {\rm ker}(\psi)$ by the isomorphism
\[(s_1,s_2,\dots,s_k)\in\Z^{n}\oplus (\Z^{n})^{k-1} \rightarrow \left(s_{1} u, \sum_{j=2}^{k}s_{j}W_{j2} w,\dots,\sum_{j=2}^k s_j W_{jk}w \right).\]

We show that the $GL_{n(k-1)}(\Z)$ matrix $D$ inside $\hat A$ represents multiplication by $\beta$ in ${\rm ker}(\psi)$.  Partition $D$ into $n\times n$ blocks $D_{ij}$ for $i,j = 2,\dots,k$. With $(x_2,\dots,x_k)\in {\rm ker}(\psi)$ corresponding to $(s_2,\dots,s_k)\in (\Z^n)^{k-1}$, and using the conjugacy between $\bigoplus_{i=1}^k C_R$ and $\hat A$, we have for each $i=2,\dots,k$ that
\[ \beta x_i = \sum_{j=2}^{k}s_{j}W_{j i}C_{R} w = \sum_{j=2}^{k}s_{j}\sum_{l=2}^k  D_{j l}W_{l i} w =
 \sum_{l=2}^{k} \left (\sum_{j=2}^{k}s_{j}D_{j l} \right) W_{l i}w.\]
The entries of
\[ (s_2^\prime, \dots, s_k^\prime) = \left( \sum_{j=2}^k s_j D_{j2},\dots,\sum_{j=2}^k s_j D_{jk}\right) \in (\Z^n)^{k-1}\]
are independent of $i$, and so
\[ (\beta x_2,\dots,\beta x_k) = \left(\sum_{l=2}^k s_l^\prime W_{l2}w,\dots,\sum_{l=2}^k s_l^\prime W_{lk}w\right) \in {\rm ker}(\psi).\]
If we set
\[ v_i = \begin{pmatrix} W_{2i} \\ \vdots \\ W_{k i}\end{pmatrix} w, \ i=2,\dots,k,\]
then $(\beta x_2,\dots,\beta x_k)$ is $s_2^\prime v_2 + s_3^\prime v_3 + \cdots + s_k^\prime v_k$. The $k-1$ vectors $v_2,\dots,v_k$ are eigenvectors of $D$ corresponding to $\beta$ because for each $i=2,\dots,k$ we have
\[ D v_i = \begin{pmatrix} \sum_{j=2}^k D_{2j}W_{ji}w \\ \sum_{j=2}^k D_{3j}W_{ji}w \\ \vdots \\ \sum_{j=2}^k D_{kj}W_{ji}w\end{pmatrix} = \begin{pmatrix} \sum_{j=2}^k W_{ji}C_R w \\ \sum_{j=2}^k W_{ji}C_R w \\ \vdots \\ \sum_{j=2}^k W_{ji}C_R w\end{pmatrix} = \begin{pmatrix} \beta \sum_{j=2}^k W_{ji}w \\ \beta \sum_{j=2}^k W_{ji}w \\ \vdots \\ \beta \sum_{j=2}^k W_{ji}w\end{pmatrix}  = \beta v_i.\]

We show that there is only one linear dependence relation on the set $v_1,v_2,\dots,v_k$, where
\[ v_1 = \begin{pmatrix} W_{21} \\ \vdots \\ W_{k1}\end{pmatrix} w,\]
and it is given by
\[\sum_{i=1}^{k}a_{i}v_{i}=\sum_{i=1}^{k}(W_{2\, i},\cdots,W_{k\, i})^{t}X_{i}u=0.\]
Suppose $\sum_{i=1}^{k}b_{i}v_{i}=0$. Multiplying this on the left by the matrix $[M_{j 2} \cdots M_{j k}]$, for any $j=1,\dots,k$, gives
\[b_{j}(I-X_{j}Y_{j})w-\sum_{i \neq j} b_{i}X_{j}Y_{i} w=0.\]
This simplifies to
\[b_{j} w=X_{j}\left(\sum_{i=1}^{k} b_{i}Y_{i}\right)w\]
for every $j$. Because
$X_{j} u=a_{j}w$ and $X_{j}$ is nonsingular, we have
\[\frac{b_{j}}{a_{j}} u= \left(\sum_{i=1}^{k} b_{i}Y_{i}\right)w,\]
for every $j$. Hence $b_j/a_j$ is a constant.

\begin{remark}{\rm In the case $k=2$, we may identify the kernel of $\psi$ with a fractional $R$-ideal, i.e., ${\rm ker}(\psi)=\{y(a_{2},-a_{1}): y \in (R:I)\}$.
}\end{remark}

The automorphism $\hat A$ induces, under the identifications defined above, an automorphism $\hat A^{*}$ of $I \oplus {\rm ker}(\psi)$, again only as an automorphism of an abelian group, but the relation of its action with the module structure goes far beyond that. For every
\[ \left(s_{1} u, \sum_{j=2}^{k}s_{j}W_{j 2} w, \dots,\sum_{j=2}^{k}s_{j}W_{j k} w \right)=(t,x_2,\dots,x_k) \in I \oplus {\rm ker}(\psi),\]
we have
\[\hat A^{*}(t,x_2,\dots,x_k)=(\beta t, \beta x_2+ \theta_{\beta}(t),\dots,\beta x_k + \theta_\beta(t) ),\]
where $\theta_{\beta}:I \rightarrow {\rm ker}(\psi)$ is explicitly given by
\[\theta_{\beta}(t)=-(s_{1} (A-\beta I_{n})Y_{1}w).\] 
Each element $y \in R$ is written as a polynomial $p(\beta)$, with rational coefficients, such that $p(C_R) \in M_{k n}(\Z)$. Hence, each $\theta_{\beta}$ determines a mapping $\mathcal{L}:R \rightarrow {\rm End}_{\Z}(I \oplus {\rm ker}(\psi))$
defined by
\[ \mathcal{L}(y)(t,x_2,\dots,x_k)=(y t, yx_2+\theta_{y}(t), \dots, yx_k + \theta_y(t) ),\]
where $\theta_{y}$ is the image of $y$ under $\theta:R \rightarrow {\rm Hom}_{\Z}(I, {\rm ker}(\psi))$, a homomorphism of abelian groups satisfying, for every $y,z \in R$,
\[\theta_{y z}(t)=\theta_{y}( z t)+y\theta_{z}(t),\] or, using the symmetry given by commutativity,
\[\theta_{y}(z t)-z \theta_{y}(t)=\theta_{z}(y t)-y \theta_{z}(t), \qquad \forall y, z \in R\,\, \forall t \in I.\]

\subsection{Another Semiconjugacy by Linear Embedding} We construct a conjugacy between $\bigoplus_{i=1}^k A^t$ and
\[ \hat C_R^t = \begin{pmatrix} C_R^t & T \\ 0 & E \end{pmatrix} \in GL_{nk}(\Z).\]
By Proposition \ref{nkconjugacy}, this implies the existence of a semiconjugacy from $C_R^t$ to $\bigoplus_{i=1}^k A^t$ by linear embedding, but not necessarily the existence of a semiconjugacy from $C_R$ to $\bigoplus_{i=1}^k A$ by a linear embedding.

To get the conjugacy we make use of any monogenic order $R_0$ contained in $R$. One possibility for $R_0$ is $\Z[\beta]$. We refer to the paper \cite{MS2} for a proof of the results on ideals that follow. We have $I (R_{0}:I)=(R_{0}:R)$, and so
\[I=((R_{0}:R):(R_{0}:I)).\]
Because $I=a_1R+\cdots+a_kR$, there is an $R$-module epimorphism
\[\psi^{*}:\bigoplus_{i=1}^{k} (R_{0}:I)\rightarrow (R_{0}:R) {\rm \ given\ by\ }\psi^{*}(y_1,\dots,y_k)=\sum_{i=1}^{ k}a_{i}y_{i},\]
and, as before, the exact sequence
\[0\rightarrow {\rm ker}(\psi^{*})\rightarrow \bigoplus_{i=1}^{k} (R_{0}:I)\rightarrow (R_{0}:R)\rightarrow 0\]
splits as a sequence of abelian groups, giving an isomorphism
\[ \bigoplus_{i=1}^k (R_0:I) \to (R_0:R)\oplus {\rm ker}(\psi^*).\]
Moreover, there are $u^{*}=(u^*_1,\dots,u^*_n)^t$ and $w^{*}=(w^*_1,\dots,w^*_n)^t$ such that
\[(R_{0}:I)=u_{1}^{*}\Z+\cdots+u_{n}^{*}\Z,\ (R_{0}:R)=w_{1}^{*}\Z+\cdots+w_{n}^{*}\Z,\]
where $A^{t}u^{*}=\beta u^{*}$ and $C_R^{t}w^{*}=\beta w^{*}$. Repeating the earlier construction, we obtain the conjugacy
\[ \left(\bigoplus_{i=1}^{k} A^{t}\right)\circ N= N\circ \hat C_R^t\]
for some $N\in GL_{nk}(\Z)$.

%Transposing this conjugacy gives $N^t(\hat C_R^t)^t = (\bigoplus_{i=1}^k A) N^t$, where
%\[ (\hat C_R^t)^t = \begin{pmatrix} C_R & 0 \\ T^t & E^t\end{pmatrix}.\]

%Although this does not give the same kind of block conjugacy, with  $A$ and $C_{R}$ in reversed roles, as before, from the transpose equation
%\[(\oplus_{i=1}^{k} A) N^{t}=N^{t}\hat C_R\]
%we may still state that there is an embedding $S$ of $\T^{n}$ into $(\T^{n})^{k}$, obtained explicitly by multiplying $C_{R}$ with the first column of the inverse of $C_R^{t}$, such that  the following diagram commutes
%\[\begin{array}{ccccc} \T^{n} & \stackrel{V^{t} \circ S}{\rightarrow} & (\T^{n})^{k} & \stackrel{\oplus A}{\rightarrow} & (\T^{n})^{k}\\\\
%\searrow  S &  (\T^{n})^{k} & \stackrel{C_{R}\times 0 \times \cdots \times 0}{\rightarrow} & (\T^{n})^{k} & \nearrow V \end{array} \]

\end{document}